\documentclass[a4paper,11pt]{article}
\usepackage{geometry}
\usepackage[utf8]{inputenc}
\usepackage[english]{babel}

\usepackage{authblk}

\usepackage{dsfont} 
\usepackage{amsmath,amssymb,amsthm}
\usepackage[justification=centering]{caption}
\usepackage[colorlinks=true, citecolor=red, linkcolor=blue]{hyperref}
\usepackage[nameinlink, capitalize]{cleveref}

\usepackage[colorinlistoftodos]{todonotes}
\usepackage{varwidth}

\usepackage{color}
\usepackage{float}
\usepackage[super]{nth}
\usepackage{booktabs}
\usepackage{rotating}
\usepackage{verbatim}
\usepackage{multirow}
\usepackage{physics}
\usepackage{mathtools}
\usepackage{subcaption}
\usepackage{bm}
\usepackage[toc,page]{appendix}
\usepackage[shortlabels]{enumitem}
\usepackage{nomencl}

\newcounter{mnotecount}[section]

\def\eps {{\epsilon}}

\theoremstyle{plain}
\newtheorem{theorem}{Theorem}
\newtheorem{corollary}{Corollary}
\newtheorem{lemma}{Lemma}
\newtheorem{proposition}{Proposition}

\theoremstyle{definition}

\newtheorem{example}[theorem]{Example}
\newtheorem{definition}{Definition}
\newtheorem{remark}{Remark}

\begin{document}

\title{Quantum Velocity Averaging (or Not\dots)\\ in the Semiclassical Regime}

\author{François Golse$^1$ and Jakob Möller$^2$\\
$^1$\'Ecole polytechnique \& IP Paris, CMLS\\
F-91128 Palaiseau Cedex, France\\
\texttt{francois.golse@polytechnique.edu}\\
\vskip .3cm
$^2$Universität Wien \& Wolgang Pauli Institute\\
O. Morgenstern Platz 1, A-1090 Wien, Austria\\
\texttt{jakob.moeller@univie.ac.at}}

\date{\today}

\maketitle

\begin{abstract}
This paper discusses the possibility of applying the velocity averaging theorems in [F. Golse, P.-L. Lions, B. Perthame, R. Sentis: J. Funct. Anal. 76(1):110--125, 1988] to the Wigner equation governing 
the quantum evolution of the Wigner transform of quantum density operators. Our first main results address the case of the Wigner function of a special class of density operators associated to mixed 
states, whose Hilbert-Schmidt norm is of order $\hbar^{d/2}$, where $d$ is the space dimension and where $\hbar$ is the reduced Planck constant. We obtain estimates which are uniform in $\hbar$ in 
the semiclassical limit $\hbar \rightarrow 0$. In space dimension $d=1$, we prove that the density function belongs to the Sobolev space $H^s(\mathbb R)$ for some $s>0$. In the case of pure states, 
we first obtain a characterization of the Wigner transform of rank-one quantum density operators, and apply this characterization (1) to analyze a rather general setting in which velocity averaging cannot 
apply to the Wigner functions of a family of rank-one density operators whose evolution is governed by the von Neumann equation, and (2) to obtain a quick derivation of Madelung's system of quantum 
hydrodynamic equations. This derivation provides a physical explanation of one key assumption used in the proof of the negative result (1) described above.
\bigskip


\end{abstract}


\section{Introduction}


The free transport, or advection operator $\partial_t+c\partial_x$ is the prototype of hyperbolic differential operators, together with the d'Alembert, or wave operator 
\[
\partial^2_t-c^2\partial^2_x=(\partial_t+c\partial_x)(\partial_t-c\partial_x).
\]
The method of characteristics shows that singularities of the initial data propagate in the solution of the Cauchy problem for the free transport 
equation, which cannot be more regular than its data.

Kinetic models for gases and plasmas systematically involve the free transport operator $\partial_t+v\cdot\nabla_x$ on $\mathbb R_t\times\mathbb R^d_x\times\mathbb R^d_v$, with a velocity $v\in\mathbb R^d$
which is itself a variable. In other words, the (main) unknown function in the kinetic theory of gases is $f\equiv f(t,x,v)$, the phase space density at time $t$ of the number of gas molecules at the position $x$ with
velocity $v$. All kinetic equations for gases or plasmas take the form 
\[
(\partial_t+v\cdot\nabla_x)f(t,x,v)=S[f](t,x,f)
\]
where $S[f]$ is a functional of the unknown $f$, such as the collision integral in the case of the Boltzmann or the Landau equations, or a term of the form $S[f]=\nabla_v\cdot (fF[f])$, where $F[f]$ is a mean-field,
self-consistent acceleration created by the unknown distribution $f$ itself, such as gravity in the case of the Vlasov-Poisson equation used in cosmology, or the electromagnetic force in the case of the 
Vlasov-Maxwell equation used in plasma physics.

If $S[f]\in L^2_{loc}(\mathbb R\times\mathbb R^d\times\mathbb R^d)$ and if $\tau+v\cdot\xi\not=0$, then $(t,x)\mapsto f(t,x,v)$ is microlocally in the Sobolev space $H^1$ in some conical neighborhood of $(\tau,\xi)$
around each $(t,x)\in\mathbb R\times\mathbb R^d$. On the other hand, if $f$ itself belongs to $L^2_{loc}(\mathbb R\times\mathbb R^d\times\mathbb R^d)$, the set of velocities $v\in\mathbb R^d$ such that 
$\tau+v\cdot\xi=0$ is an affine plane in $\mathbb R^d$, so that its Lebesgue measure is equal to $0$. Therefore, one can expect that there exists $s>0$ such that
\[
\int_{|v|\le R}f(t,x,v)\dd v\in H^s_{loc}(\mathbb R\times\mathbb R^d)
\]
(where $H^s$ designates the Sobolev space of functions with derivatives of order $\le s$ in $L^2$) for all $R>0$. In other words, averages in $v$ of $f$ are more regular in $t,x$ than $f$ itself. 

They were first  studied independently in \cite{agoshkov1984} and \cite{bgp1985} in the $L^2$ case by studying the Fourier transform of $\partial_t +v\cdot \nabla$, and more systematically in \cite{golse1988regularity}. 
The subsequent literature is rich: the $L^p$ case was studied more completely e.g. in \cite{diperna1991Lp, bezard1994regularite, arsenio2019maximal}. Velocity averaging fails in $L^1$ in general, due to possible 
concentrations in the $v$ variable, which offset the benefits of averaging: see for instance Example 1 on pp. 123--124 in \cite{golse1988regularity}. We shall see later that this example is important in the present paper. 
The concentration in $v$ is essentially the only obstruction to velocity averaging in $L^1$. Indeed, under some additional condition of equiintegrability in $\xi$ on $\{f\}$, it was proved in \cite{golse2002velocity}  that 
velocity averaging yields strong compactness of averages in $v$ of families $\{f\}$ that are bounded in $L^1_{x,v}$ and such that $\{v\cdot\nabla_xf\}$ is also bounded in $L^1_{x,v}$. Very recently, new methods have 
been used to prove new velocity averaging lemmas: see for instance \cite{arsenio2021energy, jabin2022commutator}. 

The interested reader will find a rather complete survey of the main results in this direction published before 2000 in chapter 1 of \cite{bouchut2000kinetic}.

Velocity averaging lemmas have been very useful in the study of kinetic equations because they provide strong compactness in $L^1_{loc}$, and therefore justify passing to the limit in certain nonlinearities.
This is a key step in the construction of global weak solutions for the Boltzmann \cite{DPL}, and the Vlasov-Maxwell equations \cite{diperna1989global}, for instance.

\bigskip
On the other hand, the Wigner transform $W_{\hbar}$ of quantum density operators, first introduced in \cite{wigner1997quantum}, is the quantum mechanical analogue of the phase space density in classical kinetic 
equations. In particular it has the correct macroscopic quantities (or observable densities). For instance, the particle density is the zeroth order moment of the Wigner function in the $\xi$-variable, the current density 
is its first order moment and the energy density its second order moment. However, an important caveat is that the Wigner transform takes negative values in general --- coherent states, i.e. Gaussian wave packets 
being the only pure states whose Wigner transform is positive \cite{piquet,hudson,marko-ringho-schmei}. Nevertheless, the Wigner function $W_{\hbar}$ obeys the Wigner equation, which is the quantum analogue 
to classical kinetic equations for the phase space density. In the case of point particles accelerated by an external potential $V$, the Wigner function $W_{\hbar}$ solves
\[
     \partial_tW_{\hbar} +\xi \cdot \nabla_x W_{\hbar} - \theta[V]W_{\hbar} = 0,
\]
where $\theta[V]$ is a pseudodifferential operator applied to $W_{\hbar}$, whose expression is recalled in the next section. This suggests that velocity averaging results applied to $W_{\hbar}$ should provide valuable 
information on observable densities, i.e. on moments in $\xi$ of the Wigner function $W_{\hbar}$. For instance, one could think of using velocity averaging in the context of the classical limit of quantum dynamics, e.g. 
to pass to the limit in the nonlinearities which appear in the system of Madelung equations governing the limiting density function and current density.

 In order to pass to the limit $\hbar \rightarrow 0$ it is necessary to obtain estimates which are uniform in $\hbar$. An averaging lemma that gives such estimates is called \emph{semiclassical averaging lemma}, 
 cf. \cite{GolseMauserMoller}, as opposed to estimates that might blow up when $\hbar \rightarrow 0$ (such averaging lemmas are called \emph{quantum averaging lemmas}). We prove such a result for the Wigner 
 equation with potential in Theorem \ref{thm:VAWignerL2} for the Wigner function of a special class of mixed states.

The semiclassical limit of Madelung's equations has been studied in the case of the defocussing, cubic
nonlinear Schr\"odinger equation in space dimension one in \cite{JinLvrmrMcLaughlin} by the methods of integrable systems. Velocity averaging does not play any role in \cite{JinLvrmrMcLaughlin}, and we shall see
below that there is a serious objection to using velocity averaging on families of pure quantum states in the semiclassical regime.

\smallskip
The purpose of this paper is to understand how the methods of velocity averaging can be used in the context of quantum dynamics sketched above, where one has to distinguish between quantum and semiclassical 
velocity averaging on the one hand and pure and mixed states on the other hand. The outline of our discussion of this problem is as follows: section \ref{sec:Prelim} recalls some important features of the Wigner transform, used in 
the sequel. 

Next we study velocity averaging results in an $L^2$ setting adapted to the equations of quantum dynamics, i.e. the von Neumann equation governing the evolution of quantum density operators associated to quantum 
states. 

Our first main result in this direction is Theorem \ref{thm:VAWignerL2} in section \ref{sec:VAMixedStWig}, which deals with a special class of mixed states for density operators bounded in Hilbert-Schmidt norm
independently, or uniformly in the Planck constant. 

Section \ref{sec:VAInterp} studies the same problem in the special case of space dimension one with a different approach, leading to a physically more satisfying result, bearing on the ``true'' density function, instead 
of moments of the Wigner functions involving a truncation for large values of the momentum variable as in Theorem \ref{thm:VAWignerL2}. This result, which is our second main result in this paper, is stated as 
Theorem \ref{thm:VAd=1} in section \ref{sec:VAInterp}. 

Section \ref{sec:PureSt} discusses the case of pure states, left aside in sections \ref{sec:VAMixedStWig} and \ref{sec:VAInterp}, focussed on mixed states only. Proposition \ref{P-Monokinetic}, our third main result, 
shows that sequences of pures states with vanishing Bohm pressure in the classical limit, with density functions and current densities converging strongly in $L^2_{loc}$ and kinetic energy densities converging weakly 
in $L^2_{loc}$ have monokinetic Wigner measures.  (We recall that Wigner measures are limit points in the sense of distributions of the Wigner functions in the classical limit.) This concentration effect on the momentum 
variable in the Wigner measure explains why velocity averaging cannot be applied to families of pure states in the semiclassical setting. 

The key argument in the proof of Proposition \ref{P-Monokinetic} is an identity \eqref{IdentPure} which follows from a characterization of pure states reported in Lemma \ref{L-CharRk1}, which generalizes to the case 
of all space dimensions (and makes mathematically rigorous) an earlier observation in \cite{tatarskiui1983wigner} due to Tatarski\u\i. This characterization of pure states and further consequences thereof are discussed
in section \ref{sec:WignerPureSt}.

Finally, section \ref{sec:DerMadelung} explains how \eqref{IdentPure} coming from our characterization of the Wigner transforms of pure states in Lemma \ref{L-CharRk1} leads to a quick derivation of the system of 
Madelung fluid dynamical equations for the mass and current densities of pure states governed by the Schr\"odinger equation with Lipschitz continuous external potential. This clarifies the physical meaning of one 
condition used in Proposition \ref{P-Monokinetic} leading to monokinetic Wigner measures. It also addresses the vacuum issue, i.e. points where the wave function vanishes, whose treatment had been left incomplete in earlier derivations of Madelung's equations.


\section{Preliminaries}\label{sec:Prelim}


Let $\mathfrak{H}$ be a Hilbert space. The algebra of bounded operators on $\mathfrak{H}$ is denoted by $\mathcal{L}(\mathfrak{H)}$. We denote the two-sided ideal of trace-class operators on $\mathfrak{H}$ 
by $\mathcal{L}^1(\mathfrak{H})$. A \emph{density matrix} $R\in \mathcal{L}^1(\mathfrak{H})$ is a self-adjoint ($R=R^*$), nonnegative ($\langle R\psi,\psi\rangle \geq 0$ for all $\psi \in \mathfrak{H}$) operator 
with trace $\tr_{\mathfrak{H}} R = 1$. A rank-one density operator is a \emph{pure state} --- see section \ref{sec:PureSt} below.  If $\mathfrak{H} = L^2(\mathbb{R}^d)$, the operator $R$ is an integral operator 
with integral kernel $(X,Y)\mapsto R(X,Y)$ which belongs to $L^2(\mathbb{R}^d\times \mathbb{R}^d)$. For the sake of simplicity, we shall abusively denote by $R(X,Y)$ the integral kernel of the density operator
$R$ --- so that $R$ designates both an operator on $L^2(\mathbb{R}^d)$ and an element of $L^2(\mathbb{R}^d\times \mathbb{R}^d)$. Since $R=R^*$ is trace-class, it is compact and by the spectral theorem 
there is a sequence of real eigenvalues $(\lambda_j)_{j\ge 1}$ and a Hilbert basis (i.e. a complete orthonormal system) $\{\psi_j\,:\,j\ge 1\} \subset \mathfrak{H}$ such that
\begin{equation}\label{SpecDecR}
    R(X,Y) = \sum_{j=1}^{\infty} \lambda_j \psi_j(X) \overline{\psi_j(Y)}.
\end{equation}
Since $R$ is a density operator,
\begin{equation}\label{CondLambda}
\lambda_j\ge 0\quad\text{since }R=R^*\ge 0,\quad\text{ and }\sum_{j\ge 1}\lambda_j=\tr_{L^2(\mathbb R^d)}R=1.
\end{equation}

In particular, the integral kernel $R$ has the property that the map 
\begin{equation}\label{IntKerTrCl}
z\mapsto(X\mapsto R(X,X+z))\quad\text{ belongs to }C_b(\mathbb{R}^d, L^1(\mathbb{R}^d_X)),
\end{equation}
so that the restriction of the integral kernel to the diagonal, 
i.e. $R(X,X)$, is a well-defined element of $L^1(\mathbb{R}^d_X)$. The density function $\rho\in L^1(\mathbb{R}^d)$ is 
\begin{equation}\label{DefDensityFunc}
    \rho(X) := R(X,X).
\end{equation}
Let us define
\begin{equation}
    \widehat{R}(\Xi,H) := \mathcal{F}_{X\rightarrow \Xi}\overline{\mathcal{F}_{Y\rightarrow H}} [R(\cdot,\cdot)](\Xi,H)=\int_{\mathbb R^d\times\mathbb R^d}e^{-iX\cdot\Xi+iY\cdot H} {R(X,Y)} \dd X \dd Y,
\end{equation}
and set
\begin{equation}\label{DefHatRho}
    \hat{\rho}(\Xi) := \widehat{R}(\Xi,\Xi) = \int e^{-i(X-Y)\cdot \Xi} {R(X,Y)} \dd X \dd Y.
\end{equation}
(Notice that $\hat\rho$ is \emph{not} the Fourier transform of $\rho$.) We recall that the trace of $R$ is
\begin{equation}
    \tr_{L^2(\mathbb R^d)} R = \int_{\mathbb{R}^d} \rho(X) \dd X
\end{equation}

In quantum mechanics, the analogue of the Liouville equation satisfied by the phase-space density of particles in classical mechanics is the von Neumann equation
\begin{equation}\label{eq:von Neumann} 
\left\{
\begin{aligned}
{}&i\hbar \partial_t R(t) = [\mathcal H,R(t)],
\\
&R(0) = R_{0},
\end{aligned}
\right.
\end{equation}
where $\mathcal H$ is the quantum Hamiltonian
\[
\mathcal H:=-\tfrac{\hbar^2}{2m}\Delta+V,
\]
and $V$ is the multiplication operator defined by the formula $(V\psi)(x)=V(x)\psi(x)$, assuming that $x\mapsto V(x)$ is a real-valued function defined on $\mathbb R^d$ and such that 
$\mathcal H=\mathcal H^*$ on $L^2(\mathbb R^d)$. By Stone's theorem, $-i\mathcal H/\hbar$ is the generator of a unitary group $e^{-it\mathcal H/\hbar}$ on $L^2(\mathbb R^d)$, and the (generalized) 
solution of \eqref{eq:von Neumann} is
\begin{equation}
    R(t) = e^{-it\mathcal H/\hbar}R_{0} e^{it\mathcal H/\hbar}, \quad t\in \mathbb{R},
\end{equation}
for each density operator $R_0$ on $L^2(\mathbb R^d)$.

Define the Weyl variables\footnote{This terminology is not standard, and is introduced in \cite{FilbetFG} by analogy with Weyl’s quantization: see formula (1.1.9) of \cite{Lerner2011}.} corresponding to 
$X,Y$ by
\begin{equation}\label{WeylVar}
\left\{
\begin{aligned}
x=\tfrac12(X+Y),
\\
y=\tfrac1\hbar(X-Y),
\end{aligned}
\right.
\end{equation}
and consider the Weyl-variable density matrix $\tilde{R}$ via
\begin{equation}\label{eq:weyl variable density matrix}
        \tilde{R}(t,x,y):= R(t,X,Y)=R(t,x+\tfrac\hbar{2}y,x-\tfrac\hbar{2}y).    
\end{equation}
The \emph{Wigner transform} $W_{\hbar}$ of $R$ is defined as
\begin{equation}
    W_{\hbar}[R](t,x,\xi):=\tfrac1{(2\pi)^d}\mathcal{F}_{y\rightarrow \xi} [\tilde{R}](t,x,\xi)=\tfrac{1}{(2\pi)^d} \int_{\mathbb R^d} e^{-i\xi\cdot y} \tilde{R}(t,x,y)  \dd y,
\end{equation}
so that
\begin{equation}\label{eq:inverse relation}
    \tilde{R}(t,x,y) = (2\pi)^d\mathcal{F}^{-1}_{\xi \rightarrow y}[W_{\hbar}](t,x,y).
\end{equation}
Therefore we can define the moments of the Wigner function in the $\xi$ variable, such as the \emph{particle density}
\begin{equation}
\label{eq:restriction density}
    \rho_{\hbar}(t,x) := R(t,x,x) = \tilde{R}(t,x,y)\Big|_{y=0} = \int_{\mathbb R^d}W_{\hbar}(t,x,\xi)\dd \xi,
\end{equation}
the \emph{current density}
\begin{equation}\label{eq:restriction current}
    J_{\hbar}(t,x) := -\tfrac{i}{m}\nabla_y\tilde{R}(t,x,y)\Big|_{y=0} = \tfrac1m\int_{\mathbb R^d}\xi W_{\hbar}(t,x,\xi)\dd \xi, 
\end{equation}
and the \emph{kinetic energy density}
\begin{equation}\label{eq:restriction energy}
    \mathcal{E}_{\hbar}(t,x) := -\tfrac1{2m}\Delta_y^2\tilde{R}(t,x,y)\Big|_{y=0} = \tfrac1{2m}\int_{\mathbb R^d}|\xi|^2 W_{\hbar}(t,x,\xi)\dd \xi.
\end{equation}

\smallskip
Recall the following definition:
\begin{definition} A family of measures $\rho_{\varepsilon}$ on $\mathbb{R}^d$, indexed by $\varepsilon\in(0,1]$, is \emph{tight} if,
\begin{equation}
    \sup_{0<\varepsilon\leq 1}\int_{|x|\geq R}\rho_{\varepsilon}(\dd x) \to 0 \quad \text{as } R\to \infty.
\end{equation}
In the present paper, the parameter $\varepsilon$ is the Planck constant $\hbar$. Some authors define a family $\{\psi_{\varepsilon}\,:\,\varepsilon\in(0,1]\}$ as being \emph{compact at infinity} if the measure 
with density $\rho_{\varepsilon}:=|\psi_{\varepsilon}|^2$ (with respect to the Lebesgue measure on $\mathbb R^d$) is tight.
\end{definition}

The Wigner transform enjoys the following properties, see Theorem III.2 in \cite{lions1993mesures}:
\begin{proposition}
\label{prop:properties wigner}
Let $R_\hbar\in \mathcal{L}^1(L^2(\mathbb{R}^d))$ be a family of density operators indexed by $\hbar\in(0,1]$ with density functions $\rho_{\hbar} \in L^1(\mathbb{R}^d)$. Then 
\begin{enumerate}[label=(\roman*)]
        
        \item\label{W bounded} The family $\{W_{\hbar}[R]\,:\,0<\hbar\le 1\}$ is bounded in $\mathcal{S}'(\mathbb{R}^d_x\times \mathbb{R}^d_{\xi})$; in particular there exists a subsequence $\hbar_k$ such 
        that $W_{\hbar_k}[R]$ converges to a \emph{Wigner measure} $W[R]$ in $\mathcal{S}'(\mathbb{R}^d_x\times \mathbb{R}^d_{\xi})$ as $\hbar_k \rightarrow 0$.

        \item\label{W positive} The {Wigner measure} $W$ is a bounded, nonnegative Radon measure on $\mathbb{R}^d_x\times \mathbb{R}^d_{\xi}$. 
        
        \item\label{W and rho} If the families $\{\rho_{\hbar}(x)\,:\,0<\hbar\le 1\}$ and $\{\hbar^{-d}\widehat{\rho}_{\hbar}(\xi/\hbar)\,:\,0<\hbar\le 1\}$, where $\widehat{\rho}_\hbar$ is defined in \eqref{DefHatRho}, 
        are tight and if $\rho_{\hbar}\to\rho$ vaguely as $\hbar\to 0$ (in the sense of Radon measures on $\mathbb R^d$), then
    	\begin{equation}\label{equality wigner measure density}
        	\rho = \int_{\mathbb R^d}W(\cdot,\xi)\dd \xi,
        \end{equation}
    	and
    	\begin{equation}\label{limit wigner integral}
	\int_{\mathbb{R}^d\times \mathbb{R}^d}W(x,\xi) \dd x \dd \xi = \lim_{\hbar\rightarrow 0} \int_{\mathbb{R}^d} \rho_{\hbar}(x) \dd x.
    	\end{equation}
    
\end{enumerate}  
\end{proposition}

\medskip
The Weyl-variable density matrix $\tilde{R}$, defined in \eqref{eq:weyl variable density matrix}, obeys the von Neumann equation in the form
\begin{align}
    {}&\partial_t \tilde{R}_\hbar + \tfrac1m\nabla_x\cdot(-i\nabla_y)\tilde{R}_\hbar -\delta[V] \tilde{R}_\hbar = 0,\label{eq:von Neumann R tilde}
    \\
    &\tilde{R}_\hbar(0,x,y) = \tilde{R}_{\hbar,0}(x,y),\label{eq:von Neumann R tilde data}
\end{align}
where
\begin{equation}\label{eq:def delta}
	\delta[V](x,y) := \frac{V(x+\frac{\hbar}{2}y)-V(x-\frac{\hbar}{2}y)}{i\hbar}.
\end{equation}
By taking the Fourier transform in $y$ of \eqref{eq:von Neumann R tilde}-\eqref{eq:von Neumann R tilde data} one sees that the Wigner transform $W_{\hbar}:=(2\pi)^{-d}\mathcal F_{y\to\xi}\tilde R_\hbar$
solves the \emph{Wigner equation}
\begin{align}
    {}&\partial_tW_{\hbar} +\tfrac1m\xi \cdot \nabla_x W_{\hbar} - \theta[V]W_{\hbar} = 0,\label{eq:wigner equation} 
    \\
    &W_{\hbar}(0,x,\xi) = W_{\hbar,0}(x,\xi),\label{eq:wigner equation data} 
\end{align}
where
\begin{equation}
    W_{\hbar,0}(x,\xi) := (2\pi)^{-d}\mathcal{F}_{y\rightarrow \xi}\tilde{R}_0(x,\xi),
\end{equation}
and where
\begin{equation}\label{eq:def theta V}
    \theta[V]W_{\hbar}(x,\xi) := \tfrac{1}{(2\pi)^d}\int_{\mathbb R^d\times\mathbb R^d}e^{-i(\xi-\eta) \cdot y}\delta[V](x,y)W_{\hbar}(x,\eta) \dd y \dd\eta.    
\end{equation}
Note that this can be rewritten as a convolution in the $\xi$-variable. Let
\begin{equation}\label{eq:def K[V]}
    K[V](x,\xi):= \tfrac{1}{(2\pi)^d}\int_{\mathbb R^d}e^{-i\xi \cdot y}\delta[V](x,y)\dd y;
\end{equation}
with this definition 
\begin{equation}\label{FlaTheta[V]}
     \theta[V]W_{\hbar} = K[V]\star_{\xi}W_{\hbar}.
\end{equation}
This structure will be exploited in the proof of the averaging lemma for $W_{\hbar}$ below.


\section{Semiclassical and quantum velocity averaging for mixed states}\label{sec:VAMixedStWig}


To begin with, let us recall the basic velocity averaging theorem of interest in connection with the Wigner equation. This is an extension due to DiPerna and Lions \cite{diperna1989global} of the main result 
in \cite{golse1988regularity}, which is particularly well suited to handle Vlasov-type equations --- and was used to prove the global existence of weak solutions to the Cauchy problem for the Vlasov-Maxwell
system, for all square-integrable initial data with finite mass and energy.

\begin{theorem}\label{thm: diperna lions}
Let $n\ge 0$ and $f \in L^2(\mathbb{R}\times \mathbb{R}^d_x \times \mathbb{R}^d_{\xi})$ satisfy 
\begin{equation}
	(\partial_t +\xi \cdot \nabla_x)f = g \quad \text{in }\mathcal{D}'(\mathbb{R}\times \mathbb{R}^d_x \times \mathbb{R}^d_{\xi})
\end{equation}
where $g \in L^2(\mathbb{R} \times \mathbb{R}^d_x, H^{-n}(\mathbb{R}^d_{\xi}))$.  
Then for each $\psi\in\mathcal S(\mathbb{R}^d)$,
 \begin{equation}
 	\int_{\mathbb R^d}f(\cdot,\cdot,\xi) \psi(\xi) \dd \xi \in H^{s}(\mathbb{R}\times \mathbb{R}^d),
\end{equation}
where 
\begin{equation}
    s=\frac{1}{2(n+1)}.
\end{equation}
\end{theorem}

\smallskip
In order to apply Theorem \ref{thm: diperna lions} to the Wigner equation, we first need an $L^2$ bound on the Wigner transforms of the family of density operators under consideration. Since we are ultimately interested in ``semiclassical velocity averaging" – where $\hbar\to 0$ instead of being a
physical constant – this $L^2$ bound
should in particular be uniform in the ``small parameter'' $\hbar$. It has been observed in the semiclassical analysis of the Wigner-Poisson equation in \cite{lions1993mesures,markowich1993classical},  that the $L^2$ norm of the Wigner transform 
$W_{\hbar}[R]$ of a density matrix $R$ is bounded independently of $\hbar$ if and only if $R$ belongs a certain class of mixed states, see for instance Theorem IV.5 in \cite{lions1993mesures}. Indeed, by
Plancherel's theorem
\[
\int_{\mathbb R^d}|W_\hbar[R](x,\xi)|^2\dd\xi=\tfrac1{(2\pi)^d}\int_{\mathbb R^d}|\tilde R(x,y)|^2\dd y,
\]
so that
\begin{equation}\label{WignerL2Bound}
\begin{aligned}
\iint_{\mathbb R^d\times\mathbb R^d}|W_\hbar[R](x,\xi)|^2\dd x\dd\xi=&\tfrac1{(2\pi)^d}\iint_{\mathbb R^d\times\mathbb R^d}|\tilde R(x,y)|^2\dd x\dd y
\\
=&\tfrac1{(2\pi\hbar)^d}\iint_{\mathbb R^d\times\mathbb R^d}|R(X,Y)|^2\dd X\dd Y,
\end{aligned}
\end{equation}
where the last equality follows from the change of variables \eqref{WeylVar}. On the other hand, \eqref{SpecDecR} implies that
\begin{equation}\label{TrR2}
\tr_{L^2(\mathbb R^d)}(R^2)=\iint_{\mathbb R^d\times\mathbb R^d}|R(X,Y)|^2\dd X\dd Y=\sum_{j\ge 1}\lambda_j^2.
\end{equation}
Applying these formulas to a family $\{R_\hbar\,:\,0<\hbar\le 1\}$ of density operators on $L^2(\mathbb R^d)$, whose eigenvalues are denoted by $\lambda_{\hbar,j}$, one finds that 
\[
\sum_{j\ge 1}\lambda_{\hbar,j}^2\le (2\pi\hbar)^d\sup_{0<\hbar\le 1}\|W_\hbar[R_\hbar]\|_{L^2}^2.
\]
We recall from \eqref{CondLambda} that
\[
\lambda_{\hbar,j}\ge 0\quad\text{ and }\quad\sum_{j\ge 1}\lambda_{\hbar,j}=1.
\]
By the Cauchy-Schwarz inequality
\[
1=\left(\sum_{j\ge 1}\lambda_{\hbar,j}\right)^2\le\sum_{j\ge 1}\mathbf 1_{\lambda_{\hbar,j}>0}\sum_{j\ge 1}\lambda_{\hbar,j}^2\le\text{rank}R_\hbar\cdot(2\pi\hbar)^d\sup_{0<\hbar\le 1}\|W_\hbar[R_\hbar]\|_{L^2}^2,
\]
Therefore, assuming that a family $\{R_\hbar\,:\,0<\hbar\le 1\}$ of density operators on $L^2(\mathbb R^d)$ satisfies
\[
\sup_{0<\hbar\le 1}\|W_\hbar[R_\hbar]\|_{L^2}^2=C<\infty
\]
implies that 
\[
\text{rank}R_\hbar\ge\frac1{(2\pi\hbar)^dC},\qquad\text{ so that }\varliminf_{\hbar\to 0}\left(\hbar^d\text{rank}R_\hbar\right)>0.
\]
In particular, this assumption rules out the possibility that $\{R_\hbar\,:\,0<\hbar\le 1\}$ is a family of pure states (i.e. rank-one density operators).

The following theorem compiles two results: A \emph{quantum averaging lemma}, where the constant on the right hand side depends on $\hbar$ in such a way that it blows up as $\hbar \rightarrow 0$ and a \emph{semiclassical averaging lemma}, where the constant on the right hand side is independent of $\hbar$. It is important to note that in the semiclassical setting the gain in regularity is smaller than in the purely quantum case ($H^{1/4}$ as opposed to $H^{1/2}$).

\begin{theorem}\label{thm:VAWignerL2}
Let $V\equiv V(x)$ be a real-valued continuous function defined on $\mathbb R^d$, and let  $\{R_\hbar\,:\,\hbar\in(0,1]\}$ be a family of time-dependent density operators $t\mapsto R_\hbar(t)$ defined 
for all $t\in[-T,T]$ (with $T>0$) which are continuous on $\mathbb R$ for the weak operator topology, are weak solutions of the von Neumann equation
\[
	i\hbar \partial_t R_\hbar(t)= [-\tfrac{\hbar^2}{2m}\Delta + V,R_\hbar(t)],
\]
and satisfy the bound
\[
	\sup_{|t|\le T}\tr_{L^2(\mathbb R^d)}(R_\hbar(t)^2)\le C^2(2\pi\hbar)^d
\]
for some $C>0$. For each $\psi\in\mathcal S(\mathbb R^d)$, set
\[
	\rho_\psi[R_\hbar](t,x):=\int_{\mathbb R^d}W_\hbar[R_\hbar](t,x,\xi)\psi(\xi)\dd\xi.
\]
(1) If $V\in L^\infty(\mathbb R^d)$, for each $T>0$, there exists $C'_T>0$ such that
\[
	\sup_{0<\hbar\le 1}\|\rho_\psi[R_\hbar]\|_{H^{1/2}((-T,T)\times\mathbb R^d)}\le C'_T\|V\|_{L^\infty(\mathbb R^d)}\hbar^{-1/2}.
\]
(2) If $V$ is Lipschitz-continuous on $\mathbb R^d$, for each $T>0$,
\[
	\sup_{0<\hbar\le 1}\|\rho_\psi[R_\hbar]\|_{H^{1/4}((-T,T)\times\mathbb R^d)}<\infty.
\]
\end{theorem}

\begin{proof}
First, \eqref{WignerL2Bound} and \eqref{TrR2} imply that 
\begin{equation}\label{BoundWhbar}
\|W_\hbar[R_\hbar(t)]\|^2_{L^2(\mathbb R^d\times\mathbb R^d)}\le C
\end{equation}
uniformly in $\hbar\in(0,1]$ and in $t\in[-T,T]$. On the other hand
\[
(\partial_t+\tfrac1m\xi\cdot\nabla_x)W_\hbar[R_\hbar]=\theta[V]W_\hbar[R_\hbar]=K[V]\star_\xi W_\hbar[R_\hbar]
\]
by \eqref{eq:wigner equation} and \eqref{FlaTheta[V]}.

By Plancherel's theorem,
\[
\begin{aligned}
\|(K[V]\star_\xi W_\hbar[R_\hbar])(t,x,\cdot)\|_{L^2(\mathbb R^d)}\le&\sup_{z\in\mathbb R^d}|\mathcal F_{\xi\to z}K[V](x,z)|\|W_\hbar[R_\hbar](t,x,\cdot)\|_{L^2(\mathbb R^d)}
\\
\le&\tfrac2\hbar\|V\|_{L^\infty(\mathbb R^d)}\|W_\hbar[R_\hbar](t,x,\cdot)\|_{L^2(\mathbb R^d)}\le\tfrac{2C}\hbar\|V\|_{L^\infty(\mathbb R^d)},
\end{aligned}
\]
since
\[
\mathcal F_{\xi\to z}K[V](x,z)=\delta[V](x,-z),
\]
so that statement (1) follows from Theorem \ref{thm: diperna lions} with $n=0$.

Next define
\begin{equation}\label{DefL[V]}
        L[V](x,\xi) := \tfrac1{(2\pi)^d}\int_{\mathbb R^d}e^{-i\xi \cdot y} \frac{y}{|y|} \frac{V(x+\frac{\hbar}{2}y)-V(x-\frac{\hbar}{2}y)}{\hbar|y|} \dd y,
\end{equation}
and observe that 
\[
\begin{aligned}
\nabla_\xi\cdot L[V](x,\xi)=&\tfrac1{(2\pi)^d}\int_{\mathbb R^d}e^{-i\xi \cdot y} \frac{-iy\cdot y}{|y|} \frac{V(x+\frac{\hbar}{2}y)-V(x-\frac{\hbar}{2}y)}{\hbar|y|} \dd y
\\
=&\tfrac1{(2\pi)^d}\int_{\mathbb R^d}e^{-i\xi \cdot y}\delta[V](x,y)\dd y=K[V](x,\xi)\,.
\end{aligned}
\]
Therefore
\begin{equation}\label{vNeqDivXi}
(\partial_t+\tfrac1m\xi\cdot\nabla_x)W_\hbar[R_\hbar]=\nabla_\xi\cdot L[V]\star_\xi W_\hbar[R_\hbar]=\nabla_\xi\cdot(L[V]\star_\xi W_\hbar[R_\hbar]).
\end{equation}
Reasoning as above, and denoting by $\mathrm{Lip}(V)$ the Lipschitz constant of $V$, we find that
\begin{equation}\label{BoundL[V]Wh}
\begin{aligned}
\|(L[V]\star_\xi W_\hbar[R_\hbar])(t,x,\cdot)\|_{L^2(\mathbb R^d)}\le&\sup_{z\in\mathbb R^d}|\mathcal F_{\xi\to z}L[V](x,z)|\|W_\hbar[R_\hbar])(t,x,\cdot)\|_{L^2(\mathbb R^d)}
\\
\le&\mathrm{Lip}(V)\|W_\hbar[R_\hbar])(t,x,\cdot)\|_{L^2(\mathbb R^d)}\le C\mathrm{Lip}(V),
\end{aligned}
\end{equation}
since
\begin{equation}\label{BoundL[V]}
|\mathcal F_{\xi\to z}L[V](x,z)|=\left|i\delta[V](x,y)\tfrac{y}{|y|^2}\right|=\tfrac{|\delta[V]|}{|y|}\le\mathrm{Lip}(V).
\end{equation}
With Theorem \ref{thm: diperna lions}, both inequalities \eqref{BoundWhbar} and \eqref{BoundL[V]Wh}, together with \eqref{vNeqDivXi}, imply statement (2) in the theorem.
\end{proof}

\begin{remark}
Observe that the proof of Theorem \ref{thm:VAWignerL2}, and especially \eqref{BoundL[V]}, is based on the inequality
\[
\|f\star g\|_{L^2(\mathbb R^d)}\le\|\mathcal F f\|_{L^\infty(\mathbb R^d)}\|g\|_{L^2(\mathbb R^d)}\,,
\]
a straightforward consequence of Plancherel's theorem which improves upon Young's convolution inequality
\[
\|f\star g\|_{L^2(\mathbb R^d)}\le\|f\|_{L^1(\mathbb R^d)}\|g\|_{L^2(\mathbb R^d)}.
\]
The latter inequality would lead to more stringent constraints on the potential $V$, in other words imposing that $V$ or $\nabla V$ belong to $\mathcal FL^1(\mathbb R^d)$ 
instead of $L^\infty(\mathbb R^d)$.
\end{remark}


\section{Semiclassical velocity averaging and regularity of the density function in dimension $d=1$}\label{sec:VAInterp}


In this section, we shall study the question of semiclassical velocity averaging in the special case of space dimension $d=1$, with a slightly different approach. As a matter of fact, we shall be dealing with the integral kernel
of the density operator directly, instead of its Wigner transform, so that the term ``velocity averaging'' is somewhat improper to designate this approach.

Our starting point is \eqref{eq:von Neumann R tilde}, recalled for the reader's convenience in the form
\[
\partial_t\tilde R(t,x,y)+\tfrac1m\partial_y(-i\partial_x)\tilde R(t,x,y)=y\frac{\delta[V](x,y)}{y}\tilde R(t,x,y),
\]
with $\delta[V](x,y)$ given by \eqref{eq:def delta}. (By comparison with \eqref{eq:von Neumann R tilde}, notice that we have recast the operator $\nabla_x\cdot(-i\nabla_y)$ as $\partial_y(-i\partial_x)$, since 
the discussion in the present section involves the partial Fourier transform in the $x$-variable, instead of the Wigner transform of $R(t)$, which is proportional to the partial Fourier transform of $\tilde R(t,x,y)$
in the $y$-variable.) Consider the Laplace transform of $\tilde R$:
\[
\Lambda\tilde R(\omega,x,y):=\int_0^\infty e^{-\omega t}\tilde R(t,x,y)\dd t.
\]
Since $\tilde R$ satisfies \eqref{eq:von Neumann R tilde}, its Laplace transform satisfies, for each $\omega>0$
\[
\tfrac1m\partial_y(-i\partial_x)\Lambda\tilde R(\omega,x,y)=y\frac{\delta[V](x,y)}{y}\Lambda\tilde R(\omega,x,y)+\tilde R^{in}(x,y)-\omega\Lambda\tilde R(\omega,x,y)
\]
One has
\[
\begin{aligned}
\tfrac1{2\pi}\iint_{\mathbb R\times\mathbb R}|\tilde R(t,x,y)|^2\dd x\dd y=&\tfrac1{2\pi\hbar}\iint_{\mathbb R\times\mathbb R}|R(t,X,Y)|^2dXdY
\\
=&\tfrac1{2\pi\hbar}\mathrm{Tr}_{L^2(\mathbb R)}\left(R^2(t)\right)=\tfrac1{2\pi\hbar}\mathrm{Tr}_{L^2(\mathbb R)}\left((R^{in})^2\right)\le C^{in},
\end{aligned}
\]
since $R(t)$ is obtained from $R^{in}$ by conjugation with the (unitary) Schr\"odinger group. Hence
\[
\|\Lambda\tilde R(\omega,\cdot,\cdot)\|_{L^2(\mathbb R\times\mathbb R)}\le\int_0^\infty e^{-\omega t}\|\tilde R(t,\cdot,\cdot)\|_{L^2(\mathbb R\times\mathbb R)}\dd t\le\frac{2\pi C^{in}}{\omega},
\]
and therefore
\[
\tfrac1m\partial_y(-i\partial_x)\Lambda\tilde R(\omega,x,y)=u_0(x,y)+yu_1(x,y)
\]
with
\[
u_1(x,y):=\tilde R^{in}(x,y)-\omega\Lambda\tilde R(\omega,x,y),\qquad u_2(x,y):=\frac{\delta[V](x,y)}{y}\Lambda\tilde R(\omega,x,y),
\]
and
\[
\|u_1\|_{L^2(\mathbb R\times\mathbb R)}\le 4\pi C^{in},\qquad\|u_2\|_{L^2(\mathbb R\times\mathbb R)}\le 2\pi C^{in}\mathrm{Lip}(V).
\]
We seek to obtain regularity in $x$ on the Laplace transform of the density function
\[
\Lambda\rho(\omega,x)=\Lambda\tilde R(\omega,x,x)=\int_\mathbb R\Lambda W_\hbar[R](\omega,x,\xi)\dd\xi
\]
Notice the difference with Theorem \ref{thm:VAWignerL2}: one seeks information on $\rho_\psi[R]$ with $\psi\equiv 1$ instead of $\psi\in\mathcal S(\mathbb R)$. This situation is a special case of the following result.

\begin{theorem}\label{thm:VAd=1}
Assume that $R=R^*\in\mathcal L^1(L^2(\mathbb R))$ is such that $\tilde R(x,y):=R(x+\tfrac\hbar{2}y,x-\tfrac\hbar{2}y)$ satisfies
\[
\partial_y(-i\partial_x)\tilde R(x,y)=\sum_{k=0}^nu_k(x,y)y^k,
\]
with $u_0,\ldots,u_n\in L^2(\mathbb R\times\mathbb R)$. Then the function $x\mapsto\rho(x):=R(x,x)=\tilde R(x,0)$ belongs to the Sobolev space $H^\frac1{2(n+1)}(\mathbb R)$, and there exists $C>0$ such that,
for all $R\in\mathcal L^1(L^2(\mathbb R))$ satisfying the assumptions above
\[
\|\rho\|_{H^\frac1{2(n+1)}(\mathbb R)}\le C\left(\mathrm{tr}_{L^2(\mathbb R)}(|R|)+\|\tilde R\|_{L^2(\mathbb R\times\mathbb R)}\sum_{k=0}^n\|b_k\|^2_{L^2(\mathbb R\times\mathbb R)}\right),
\]
where the constant $C$ is independent of $\hbar$.
\end{theorem}

\begin{proof}
Applying the Fourier transform $\mathcal F_{x\to\xi}$ to both sides of the equality satisfied by $\tilde R(x,y)$, and setting $f(\xi,y):=\mathcal F_{x\to\xi}[\tilde R(\cdot,y)](\xi)$,
we find that
\[
\|f\|_{L^2(\mathbb R\times\mathbb R)}=\sqrt{2\pi}\|\tilde R\|_{L^2(\mathbb R\times\mathbb R)},
\]
by Plancherel's theorem, while $b_k(\xi,y):=\mathcal F_{x\to\xi}[u_k(\cdot,y)](\xi)$ satisfies
\[
\|b_k\|_{L^2(\mathbb R\times\mathbb R)}=\sqrt{2\pi}\|u_k\|_{L^2(\mathbb R\times\mathbb R)}.
\]
Finally
\[
\xi\partial_yf(\xi,y)=\sum_{k=0}^nb_k(\xi,y)y^k.
\]
Set $G(z):=e^{-z^2/2}/\sqrt{2\pi}$, the centered, reduced Gaussian density on $\mathbb R$. For each $\eps>0$, write the decomposition
\[
\begin{aligned}
f(\xi,0)=&\int_0^\infty\eps G(\eps y)\left(f(\xi,y)+f(\xi,-y)+2f(\xi,0)-f(\xi,y)-f(\xi,-y)\right)\dd y
\\
=&\int_\mathbb R\eps G(\eps y)f(\xi,y)\dd y\!-\!\int_0^\infty\eps G(\eps y)\left(\int_0^y(\partial_2f(\xi,z)-\partial_2f(\xi,-z))\dd z\right)\dd y
\\
=&\int_\mathbb R\eps G(\eps y)f(\xi,y)\dd y\!-\!\frac1\xi\int_0^\infty\eps G(\eps y)\left(\int_0^y\sum_{k=0}^n\beta_k(\xi,z)z^k\dd z\right)\dd y,
\end{aligned}
\]
where
\[
\beta_k(\xi,z)=b_k(\xi,z)-(-1)^kb_k(\xi,-z),\qquad k=0,\ldots,n.
\]
Applying Fubini's theorem to the summable function $(y,z)\mapsto \eps G(\eps y)\beta_k(\xi,z)z^k\mathbf 1_{0<z<y}$ defined a.e. on $\mathbb R\times\mathbb R$, leads to 
\[
\begin{aligned}
f(\xi,0)=&\int_\mathbb R\eps G(\eps y)f(\xi,y)\dd y-\frac1\xi\sum_{k=0}^n\int_0^\infty\left(z^k\int_z^\infty\eps G(\eps y)\dd y\right)\beta_k(\xi,z)\dd z
\\
=&\int_\mathbb R\eps G(\eps y)f(\xi,y)\dd y-\frac1\xi\sum_{k=0}^n\frac1{\eps^k}\int_0^\infty\left(\eps^kz^k\int_{\eps z}^\infty G(Y)\dd Y\right)\beta_k(\xi,z)\dd z.
\end{aligned}
\]
Hence
\[
\frac{|f(\xi,0)|^2}{n+2}\le\left(\int_\mathbb R\eps G(\eps y)f(\xi,y)\dd y\right)^2
\!\!+\sum_{k=0}^n\frac1{\eps^{2k}|\xi|^2}\left(\int_0^\infty\left(\eps^kz^k\int_{\eps z}^\infty G(Y)\dd Y\right)\beta_k(\xi,z)\dd z\right)^2\!\!.
\]
By the Cauchy-Schwarz inequality,
\[
\left(\int_\mathbb R\eps G(\eps y)f(\xi,y)\dd y\right)^2\le\eps\int_\mathbb R G(Y)^2\dd Y\int_\mathbb R|f(\xi,y)|^2\dd y=\frac\eps{2\sqrt\pi}\int_\mathbb R|f(\xi,y)|^2\dd y,
\]
while
\[
\begin{aligned}
\left(\int_0^\infty\left(\eps^kz^k\int_{\eps z}^\infty G(Y)\dd Y\right)\beta_k(\xi,z)\dd z\right)^2
\\
\le\frac1\eps\int_0^\infty Z^{2k}\left(\int_Z^\infty G(Y)\dd Y\right)^2\dd Z\int_0^\infty|\beta_k(\xi,z)|^2\dd z
\\
\le\frac1\eps\int_0^\infty Z^{2k}\left(\int_Z^\infty G(Y)\dd Y\right)\dd Z\int_0^\infty|\beta_k(\xi,z)|^2\dd z
\\
=\frac1\eps\int_0^\infty\frac{Y^{2k+1}}{2k+1}G(Y)\dd Y\int_0^\infty|\beta_k(\xi,z)|^2\dd z.
\end{aligned}
\]
Setting
\[
\gamma_k:=\int_0^\infty\frac{Y^{2k+1}}{2k+1}G(Y)\dd Y=\frac1{2k+1}\frac1{\sqrt{2\pi}}\int_0^\infty (2U)^ke^{-U}\dd U=\frac{2^kk!}{(2k+1)\sqrt{2\pi}},
\]
we arrive at the bound
\[
\frac{|f(\xi,0)|^2}{n+2}\le\frac\eps{2\sqrt\pi}\|f(\xi,\cdot)\|^2_{L^2(\mathbb R)}+\sum_{k=0}^n\frac{\gamma_k}{\eps^{2k+1}|\xi|^2}\|\beta_k(\xi,\cdot)\|^2_{L^2(0,+\infty)}
\]
which holds for all $\xi\not=0$, and for all $\eps>0$. 

At this point, we consider separately the cases $|\xi|\ge 1$ and $|\xi|<1$. In the first case, we equilibrate the term of order $\eps$ and the term of highest degree in $1/\eps$, and set
\[
\eps^{2n+2}=\frac1{|\xi|^2}.
\]
Thus
\[
\begin{aligned}
\mathbf 1_{|\xi|\ge 1}\frac{|\xi|^\frac1{n+1}}{n+2}|f(\xi,0)|^2\le&\tfrac1{2\sqrt\pi}\|f(\xi,\cdot)\|^2_{L^2(\mathbb R)}
\\
&+\!\sum_{k=0}^{n-1}\frac{\mathbf 1_{|\xi|\ge 1}\gamma_k}{|\xi|^\frac{2(n-k)}{n+1}}\|\beta_k(\xi,\cdot)\|^2_{L^2(0,+\infty)}
\!+\!\gamma_n\|\beta_n(\xi,\cdot)\|^2_{L^2(0,+\infty)}
\\
\le&\tfrac1{2\sqrt\pi}\|f(\xi,\cdot)\|^2_{L^2(\mathbb R)}\!+\!\sum_{k=0}^n\gamma_k\|\beta_k(\xi,\cdot)\|^2_{L^2(0,+\infty)}\,.
\end{aligned}
\]
In the second case, i.e. $|\xi|<1$, we equilibrate the term of order $\eps$ and the term of lowest degree in $1/\eps$, and set
\[
\eps=\frac1{|\xi|},
\]
so that
\[
\begin{aligned}
\mathbf 1_{|\xi|<1}|\xi|\frac{|f(\xi,0)|^2}{n+2}\le&\tfrac1{2\sqrt\pi}\|f(\xi,\cdot)\|^2_{L^2(\mathbb R)}+\sum_{k=0}^n\gamma_k|\xi|^{2k}\mathbf 1_{|\xi|<1}\|\beta_k(\xi,\cdot)\|^2_{L^2(0,+\infty)}
\\
\le&\tfrac1{2\sqrt\pi}\|f(\xi,\cdot)\|^2_{L^2(\mathbb R)}+\sum_{k=0}^n\gamma_k\|\beta_k(\xi,\cdot)\|^2_{L^2(0,+\infty)}.
\end{aligned}
\]
Summarizing, we arrive at the inequality
\[
\begin{aligned}
\tfrac1{n+2}\min\left(|\xi|,|\xi|^\frac1{n+1}\right)|f(\xi,0)|^2\le&\tfrac1{2\sqrt\pi}\|f(\xi,\cdot)\|^2_{L^2(\mathbb R)}+\sum_{k=0}^n\gamma_k\|\beta_k(\xi,\cdot)\|^2_{L^2(0,+\infty)}
\\
\le&\tfrac1{2\sqrt\pi}\|f(\xi,\cdot)\|^2_{L^2(\mathbb R)}+\sum_{k=0}^n2\gamma_k\|b_k(\xi,\cdot)\|^2_{L^2(\mathbb R)}.
\end{aligned}
\]
Now, by Young's inequality
\[
|\xi|^\frac1{n+1}\le\frac{|\xi|}{n+1}+\frac{n}{n+1},
\]
so that
\[
\begin{aligned}
\tfrac1{n+2}|\xi|^\frac1{n+1}|f(\xi,0)|^2\le&\tfrac{(|\xi|+n)\mathbf 1_{|\xi|<1}}{(n+1)(n+2)}|f(\xi,0)|^2+\tfrac{\mathbf 1_{|\xi|\ge 1}}{(n+2)}|\xi|^\frac1{n+1}|f(\xi,0)|^2
\\
\le&\tfrac{\mathbf 1_{|\xi|<1}}{n+2}|f(\xi,0)|^2+\tfrac1{n+2}\min\left(|\xi|,|\xi|^\frac1{n+1}\right)|f(\xi,0)|^2
\\
\le&\tfrac{\mathbf 1_{|\xi|<1}}{n+2}\|\rho\|^2_{L^1(\mathbb R)}+\tfrac1{n+2}\min\left(|\xi|,|\xi|^\frac1{n+1}\right)|f(\xi,0)|^2,
\end{aligned}
\]
and hence
\[
\begin{aligned}
\tfrac1{n+2}\int_\mathbb R|\xi|^\frac1{n+1}|f(\xi,0)|^2\dd\xi
\le&\tfrac{2}{n+2}\|\rho\|^2_{L^1(\mathbb R)}+\tfrac1{2\sqrt\pi}\|f\|^2_{L^2(\mathbb R\times\mathbb R)}+\sum_{k=0}^n2\gamma_k\|b_k\|^2_{L^2(\mathbb R\times\mathbb R)}
\\
\le&\tfrac{2}{n+2}+\sqrt\pi\|\tilde R\|_{L^2(\mathbb R\times\mathbb R)}+\sum_{k=0}^n2\gamma_k\|b_k\|^2_{L^2(\mathbb R\times\mathbb R)}.
\end{aligned}
\]
By the same token
\[
\begin{aligned}
\int_\mathbb R\frac{1+|\xi|^\frac1{n+1}}{2(n+2)}|f(\xi,0)|^2\dd\xi\le&\int_\mathbb R\frac{\mathbf 1_{|\xi|<1}+|\xi|^\frac1{n+1}\mathbf 1_{|\xi|\ge 1}}{n+2}|f(\xi,0)|^2\dd\xi
\\
\le&\tfrac2{n+2}+\tfrac2{n+2}+\sqrt\pi\|\tilde R\|_{L^2(\mathbb R\times\mathbb R)}+\sum_{k=0}^n2\gamma_k\|b_k\|^2_{L^2(\mathbb R\times\mathbb R)}.
\end{aligned}
\]
This concludes the proof.
\end{proof}

\begin{remark}
Notice that Theorem \ref{thm:VAd=1} (and its proof) make use of the fact that $\rho\in L^1(\mathbb R)$, implied by the fact that $R$ is a trace-class operator, to conclude that $\rho$ belongs to $H^s(\mathbb R)$
for some $s>0$. If one assumes only that $R$ is a Hilbert-Schmidt operator, the proof only leads to a bound on
\[
\int_\mathbb R\min(|\xi|,|\xi|^{2s})|\mathcal F_{x\to\xi}\rho(\xi)|^2\dd\xi
\]
with $s=\frac1{2(n+1)}$. This control is weaker than a bound on the homogeneous Sobolev norm $\dot H^s(\mathbb R)$, except in the case $n=0$, where it is found that
\[
\|\rho\|_{\dot H^{1/2}(\mathbb R)}\le C\|\tilde R\|_{L^2(\mathbb R\times\mathbb R)}^{1/2}\|\partial_x\partial_y\tilde R\|_{L^2(\mathbb R\times\mathbb R)}^{1/2}.
\]
This is the analogue of the basic $L^2$ case of velocity averaging in space dimension $1$: for all $f\equiv f(x,v)$, with $x,v\in\mathbb R$,
\[
\left\|\int_\mathbb R fdv\right\|_{\dot H^{1/2}(\mathbb R)}\le C\|f\|_{L^2(\mathbb R\times\mathbb R)}^{1/2}\|v\partial_xf\|_{L^2(\mathbb R\times\mathbb R)}^{1/2}.
\]
In that case, no localization in $v$ is needed: see \cite{golse1988regularity}.
\end{remark}


\section{Pure states and the Wigner transform}\label{sec:WignerPureSt}


In view of the remark preceding the statement of Theorem \ref{thm:VAWignerL2}, the problem of velocity averaging for the Wigner transform of pure states cannot be addressed by the method used in 
section \ref{sec:VAMixedStWig} in the classical limit, and certainly not by using Theorem \ref{thm: diperna lions}. The case of pure states in the semiclassical regime must therefore be discussed separately. 

As a preparation, the present section discusses a few properties of Wigner functions of pure states. These properties have several important implications, especially in semiclassical velocity averaging.

We first recall the definition of a pure quantum state.

\begin{definition}
A density operator $R$ on $\mathfrak H$ defines a \emph{pure state} if it is a projection, in other words if it is an idempotent element of the algebra $\mathcal L(\mathfrak H)$, i.e. if $R^2=R$.
\end{definition}

Since a density operator is self-adjoint, $R$ defines a pure state if and only if it is an orthogonal projection, i.e. $R^2=R=R^*$. Since the rank of an orthogonal projection in $\mathfrak H$ is its trace,
and since a the trace of a density operator is $1$, a pure state corresponds to a rank-one orthogonal projection in $\mathfrak H$. Pick $\psi\in\mathfrak H\setminus\{0\}$ so that the range of the
rank-one orthogonal projection $R$ is $\mathbb C\psi$; without loss of generality, one can assume that $\|\psi\|_\mathfrak H=1$. Thus the operator $R$ can be written in Dirac's bra-ket notation as
\begin{equation}\label{Rk1bra-ket}
	R = \ketbra{\psi} ,
\end{equation}
where we recall that $|\psi\rangle$ designates the vector $\psi\in\mathfrak H$, while $\langle\psi|$ designates the continuous linear functional $\phi\mapsto(\psi|\phi)_\mathfrak H$ on $\mathfrak H$.
When $\mathfrak H=L^2(\mathbb R^d)$, this is equivalent to the fact that the pure state density operator $R$ has integral kernel
\begin{equation}\label{Rk1Kernel}
	R(X,Y) = \psi(X)\overline{\psi(Y)}.
\end{equation} 

There are several well-known characterizations of pure states in quantum mechanics, besides the condition $R^2=R$ in the definition above. For instance, a density operator represents a pure state
if and only if its von Neumann entropy $-\mathrm{tr}(R\ln R)=0$. However, for the purpose of our discussion, we are chiefly interested in characterizations of pure states involving the Wigner function
of the associated density operator. One could try to use the identity $W_\hbar[R]=W_\hbar[R^2]=W_\hbar[R]\circ_\hbar W_\hbar[R]$ where $\circ_\hbar$ designates the Moyal product. A quick glance
at the formula for the Moyal product (see for instance formula (18.5.6) in chapter XVIII of \cite{Hormander3}) suggests that this approach could prove somewhat unpractical.

We shall appeal to the following criterion. As a matter of fact, we shall not use exactly this criterion, but a consequence thereof, stated below as Corollary \ref{C-ImpliRk1}.

\begin{lemma}\label{L-CharRk1}
Let $R=R^*\in\mathcal L^1(L^2(\mathbb R^d))$ with Wigner function $W_\hbar[R](x,\xi)$. Assume that $(x,y)\mapsto\mathcal F_{\xi\to y}W_\hbar[R](x,y)$ (the partial Fourier transform of 
$W_\hbar[R]$ in the $\xi$-variable) is of class $C^1$ on $\mathbb R^d\times\mathbb R^d$ and that $\mathcal F_{\xi\to y}W_\hbar[R](x,y)\not=0$ for all $x,y\in\mathbb R^d$. 
Then $R$ is a rank-one density operator if and only if, for all $j,k=1,\ldots,d$,
\begin{equation}\label{IdentTatarskii}
\left\{
\begin{aligned}
\tfrac4{\hbar^2}\partial_{y_j}\left(\frac{\partial_{y_k}\mathcal F_{\xi\to y}W_\hbar[R](x,y)}{\mathcal F_{\xi\to y}W_\hbar[R](x,y)}\right)
&\!=\!\partial_{x_j}\left(\frac{\partial_{x_k}\mathcal F_{\xi\to y}W_\hbar[R](x,y)}{\mathcal F_{\xi\to y}W_\hbar[R](x,y)}\right),
\\
\partial_{y_j}\left(\frac{\partial_{x_k}\mathcal F_{\xi\to y}W_\hbar[R](x,y)}{\mathcal F_{\xi\to y}W_\hbar[R](x,y)}\right)
&\!=\!\partial_{x_j}\left(\frac{\partial_{y_k}\mathcal F_{\xi\to y}W_\hbar[R](x,y)}{\mathcal F_{\xi\to y}W_\hbar[R](x,y)}\right),
\end{aligned}
\right.
\quad\text{ in }\mathcal D'(\mathbb R^d\times\mathbb R^d).
\end{equation}
\end{lemma}

Notice that (a formal variant of) Lemma \ref{L-CharRk1} in the special case of space dimension $d=1$ was already stated in \cite{tatarskiui1983wigner}.

\begin{proof}
Assume first that $R=\ketbra{\psi}$, whose integral kernel is $R(X,Y)=\psi(X)\overline{\psi(Y)}$. Then 
\[
\mathcal F_{\xi\to y}W_\hbar[R](x,y)=\psi(x-\tfrac\hbar{2}y)\overline{\psi(x+\tfrac{\hbar}2y)},
\]
and our assumptions imply that $\psi\in C^1(\mathbb R^d,\mathbb C\setminus\{0\})$. Then, for $\ell=1,\ldots,d$,
\[
\left\{
\begin{aligned}
\frac{\partial_{x_\ell}\mathcal F_{\xi\to y}W_\hbar[R](x,y)}{\mathcal F_{\xi\to y}W_\hbar[R](x,y)}
&=\frac{\partial_\ell\psi(x-\tfrac\hbar{2}y)}{\psi(x-\tfrac\hbar{2}y)}+\frac{\partial_\ell\overline{\psi(x+\tfrac\hbar{2}y)}}{\overline{\psi(x+\tfrac\hbar{2}y)}},
\\
\frac{\partial_{y_\ell}\mathcal F_{\xi\to y}W_\hbar[R](x,y)}{\mathcal F_{\xi\to y}W_\hbar[R](x,y)}
&=-\tfrac\hbar{2}\frac{\partial_\ell\psi(x-\tfrac\hbar{2}y)}{\psi(x-\tfrac\hbar{2}y)}+\tfrac\hbar{2}\frac{\partial_\ell\overline{\psi(x+\tfrac\hbar{2}y)}}{\overline{\psi(x+\tfrac\hbar{2}y)}}.
\end{aligned}
\right.
\]
Set $\omega_+=\mathbb C\setminus[0,+\infty)$ and $\omega_-=\mathbb C\setminus(-\infty,0]$. Let $\log_-$ be the principal determination of the logarithm on $\omega_-$, and set $\log_+z=\log_-(-z)+i\pi$. 
One has $\omega_+\cap\omega_-=H_+\cup H_-$ where $H_\pm$ is the open upper (resp.  lower) half-plane in $\mathbb C$. One easily checks that
\[
\begin{aligned}
\log_+z-\log_-z&=0,&&\qquad z\in H_+,
\\
\log_+z-\log_-z&=2i\pi,&&\qquad z\in H_-.
\end{aligned}
\]
Set $\Omega_\pm=\psi^{-1}(\omega_\pm)$, which is an open subset of $\mathbb R^d$ since $\psi$ is continuous. One has obviously $\mathbb R^d=\Omega_+\cup\Omega_-$ since $\psi$ takes its values 
in $\mathbb C\setminus\{0\}$. (Indeed $\Omega_+\cup\Omega_-=\psi^{-1}(\omega_+\cup\omega_-)=\psi^{-1}(\mathbb C\setminus\{0\})=\mathbb R^d$.) Besides 
\[
\Omega_+\cap\Omega_-=\psi^{-1}(H_+)\cup\psi^{-1}(H_-)\quad\text{ with }\psi^{-1}(H_\pm)\text{ open and disjoint in }\mathbb R^d.
\]
Since $\log_+\psi-\log_-\psi$ is a constant multiple of $2i\pi$ on $\psi^{-1}(H_\pm)$, one has
\[
\partial_\ell\log_+\psi=\frac{\partial_\ell\psi}{\psi}=\partial_\ell\log_-\psi\quad\text{ on }\Omega_+\cap\Omega_-.
\]
Hence
\[
\left\{
\begin{aligned}
\frac{\partial_{x_\ell}\mathcal F_{\xi\to y}W_\hbar[R](x,y)}{\mathcal F_{\xi\to y}W_\hbar[R](x,y)}
&=\partial_\ell(\log_\pm\psi)(x-\tfrac\hbar{2}y)+\partial_\ell(\log_\pm\overline\psi)(x+\tfrac\hbar{2}y),
\\
\frac{\partial_{y_\ell}\mathcal F_{\xi\to y}W_\hbar[R](x,y)}{\mathcal F_{\xi\to y}W_\hbar[R](x,y)}
&=-\tfrac\hbar{2}\partial_\ell(\log_\pm\psi)(x-\tfrac\hbar{2}y)+\tfrac\hbar{2}\partial_\ell(\log_\pm\overline\psi)(x+\tfrac\hbar{2}y),
\end{aligned}
\right.
\]
where the determination $\log_+$ or $\log_-$ of the logarithm is chosen according to whether $x-\tfrac\hbar{2}y$ and $x+\tfrac\hbar{2}y$ belong to $\Omega_+$ or $\Omega_-$, so that
\[
\left\{
\begin{aligned}
\partial_{x_j}\left(\frac{\partial_{x_k}\mathcal F_{\xi\to y}W_\hbar[R](x,y)}{\mathcal F_{\xi\to y}W_\hbar[R](x,y)}\right)
&=\partial_j\partial_k(\log_\pm\psi)(x-\tfrac\hbar{2}y)+\partial_j\partial_k(\log_\pm\overline\psi)(x+\tfrac\hbar{2}y),
\\
\partial_{x_j}\left(\frac{\partial_{y_k}\mathcal F_{\xi\to y}W_\hbar[R](x,y)}{\mathcal F_{\xi\to y}W_\hbar[R](x,y)}\right)
&=-\tfrac\hbar{2}\partial_j\partial_k(\log_\pm\psi)(x-\tfrac\hbar{2}y)+\tfrac\hbar{2}\partial_j\partial_k(\log_\pm\overline\psi)(x+\tfrac\hbar{2}y),
\\
\partial_{y_j}\left(\frac{\partial_{y_k}\mathcal F_{\xi\to y}W_\hbar[R](x,y)}{\mathcal F_{\xi\to y}W_\hbar[R](x,y)}\right)
&=\tfrac{\hbar^2}{4}\partial_j\partial_k(\log_\pm\psi)(x-\tfrac\hbar{2}y)+\tfrac{\hbar^2}{4}\partial_j\partial_k(\log_\pm\overline\psi)(x+\tfrac\hbar{2}y).
\end{aligned}
\right.
\]
Comparing the first and last line in the system above gives the first $d^2$ identities in the lemma. The second group of identities in the lemma follows from the symmetry in $j,k$ of the partial derivatives
$\partial_j\partial_k(\log_\pm\psi)$ and $\partial_j\partial_k(\log_\pm\overline\psi)$ in the sense of distributions on $\Omega_\pm$. Hence the conditions in the lemma are necessary for the density operator
$R$ to be of rank equal to one.

Conversely, assume that $\tilde R:=\mathcal F_{\xi\to y}W_\hbar[R]\in C^1(\mathbb R^d\times\mathbb R^d,\mathbb C\setminus\{0\})$ satisfies the identities in the lemma, and set
\[
R(X,Y)=\tilde R\left(\tfrac{X+Y}2,\tfrac{Y-X}\hbar\right),\qquad X,Y\in\mathbb R^d,
\]
which is an integral kernel for the density operator $R$. Set $U_\pm=R^{-1}(\omega_\pm)$, which is open in $\mathbb R^d\times\mathbb R^d$ since $R$ is continuous, and satisfy 
$\mathbb R^d\times\mathbb R^d=U_+\cup U_-$, since $R$ takes its values in $\mathbb C\setminus\{0\}=\omega_+\cup\omega_-$. (Indeed $U_+\cup U_-=R^{-1}(\omega_+\cup\omega_-)
=R^{-1}(\mathbb C\setminus\{0\})=\mathbb R^d\times\mathbb R^d$.) As above, 
\[
U_+\cap U_-= R^{-1}(H_+)\cup R^{-1}(H_-)\quad\text{ with }R^{-1}(H_\pm)\text{ open and disjoint in }\mathbb R^d\times\mathbb R^d,
\]
and
\[
\nabla_{X,Y}\log_+R(X,Y)=\frac{\nabla_{X,Y}R(X,Y)}{R(X,Y)}=\nabla_{X,Y}\log_-R(X,Y)\,,\qquad(X,Y)\in U_+\cap U_-,
\]
since $\log_+R-\log_-R$ is a constant multiple of $2i\pi$ on $R^{-1}(H_\pm)$. The identities in the lemma are recast as
\[
\left\{
\begin{aligned}
\tfrac4{\hbar^2}\partial_{y_j}\partial_{y_k}\log_\pm\tilde R(x,y)&=\partial_{x_j}\partial_{x_k}\log_\pm\tilde R(x,y)
\\
\partial_{y_j}\partial_{x_k}\log_\pm\tilde R(x,y)&=\partial_{x_j}\partial_{y_k}\log_\pm\tilde R(x,y)
\end{aligned}
\right.
\qquad\text{ in }\mathcal D'(V_\pm)\qquad\text{ for }1\le j,k\le d,
\]
where $V_\pm$ is the inverse image of $U_\pm$ by the linear transformation $(x,y)\mapsto(x-\tfrac\hbar{2}y,x+\tfrac\hbar{2}y)$. Since this linear transformation is one-to-one and onto, its inverse
maps $U_+\cup U_-=\mathbb R^d\times\mathbb R^d$ onto $V_+\cup V_-=\mathbb R^d\times\mathbb R^d$. By the chain rule
\[
\left\{
\begin{aligned}
\partial_{x_\ell}\log_\pm\tilde R(x,y)&=\partial_{x_\ell}\log_\pm R(x-\tfrac\hbar{2}y,x+\tfrac\hbar{2}y)=(\partial_{X_\ell}+\partial_{Y_\ell})\log_\pm R(x-\tfrac\hbar{2}y,x+\tfrac\hbar{2}y),
\\
\partial_{y_\ell}\log_\pm\tilde R(x,y)&=\partial_{y_\ell}\log_\pm R(x-\tfrac\hbar{2}y,x+\tfrac\hbar{2}y)=\tfrac\hbar{2}(\partial_{Y_\ell}-\partial_{X_\ell})\log_\pm R(x-\tfrac\hbar{2}y,x+\tfrac\hbar{2}y),
\end{aligned}
\right.
\]
so that, for $1\le j,k\le d$
\[
\left\{
\begin{aligned}
(\partial_{Y_j}-\partial_{X_j})(\partial_{Y_k}-\partial_{X_k})\log_\pm R(X,Y)&=(\partial_{X_j}+\partial_{Y_j})(\partial_{X_k}+\partial_{Y_k})\log_\pm R(X,Y),
\\
\tfrac\hbar{2}(\partial_{Y_j}-\partial_{X_j})(\partial_{Y_k}+\partial_{X_k})\log_\pm R(X,Y)&=\tfrac\hbar{2}(\partial_{X_j}+\partial_{Y_j})(\partial_{Y_k}-\partial_{X_k})\log_\pm R(X,Y),
\end{aligned}
\right.
\]
which is recast as
\[
\left\{
\begin{aligned}
(\partial_{Y_j}\partial_{X_k}+\partial_{X_j}\partial_{Y_k})\log_\pm R(X,Y)=0,
\\
(\partial_{Y_j}\partial_{X_k}-\partial_{X_j}\partial_{Y_k})\log_\pm R(X,Y)=0,
\end{aligned}
\right.
\]
or, equivalently
\[
\partial_{X_j}\partial_{Y_k}\ln_\pm R(X,Y)=0\quad\text{ in }\mathcal D'(U_\pm)\quad\text{ for all }1\le j,k\le d.
\]
This implies that 
\[
\partial_{X_j}\left(\frac{\partial_{Y_k}R(X,Y)}{R(X,Y)}\right)=0\quad\text{ in }\mathcal D'(\mathbb R^d\times\mathbb R^d)\quad\text{ for all }1\le j,k\le d.
\]
Since $\mathbb R^d\times\mathbb R^d$ is connected, this implies the existence of $b_1,\ldots,b_d\in C(\mathbb R^d)$ such that
\[
\frac{\partial_{Y_k}R(X,Y)}{R(X,Y)}=b_k(Y)\quad\text{ for all }(X,Y)\in\mathbb R^d\times\mathbb R^d\,.
\]
Besides
\[
\partial_{Y_j}b_k(Y)=\partial_{Y_k}b_j(Y)\quad\text{ in }\mathcal D'(\mathbb R^d)\quad\text{ for all }1\le j,k\le d.
\]
Indeed
\[
b_k(Y)=\partial_{Y_k}\log_\pm R(X,Y)\quad\text{ for all }(X,Y)\in U_\pm\quad\text{ and all }k=1,\ldots,d,
\]
so that
\[
\partial_{Y_j}b_k(Y)=\partial_{Y_k}b_j(Y)\quad\text{ in }\mathcal D'(U_+\cup U_-)\quad\text{ for all }1\le j,k\le d,
\]
and $U_+\cup U_-=\mathbb R^d\times\mathbb R^d$. Hence there exists $B\in C^1(\mathbb R^d)$ such that 
\[
\nabla B(Y)=(b_1(Y),\ldots,b_d(Y))\,,\quad Y\in\mathbb R^d.
\]
Therefore
\[
\log_\pm R(X,Y)-B(Y)=\mathfrak A_\alpha^\pm(X)\quad\text{ for all }(X,Y)\in U^\pm_\alpha,
\]
where $U^\pm_\alpha$ is the family of connected components of $U_\pm$.

Exchanging the variables $X$ and $Y$, and arguing as above shows that there exists $A\in C^1(\mathbb R^d)$ so that
\[
\log_\pm R(X,Y)-A(X)=\mathfrak B_\alpha^\pm(Y)\quad\text{ for all }(X,Y)\in U^\pm_\alpha.
\]
In other words
\[
\mathfrak A_\alpha^\pm(X)+B(Y)=A(X)+\mathfrak B_\alpha(Y)\quad\text{ for all }(X,Y)\in U^\pm_\alpha,
\]
so that
\[
A(X)-\mathfrak A_\alpha^\pm(X)=B(Y)-\mathfrak B_\alpha^\pm(Y)\quad\text{ for all }(X,Y)\in U^\pm_\alpha.
\]
Therefore, there exists a family of constants $C_\alpha^\pm$ indexed by the family of connected components $U_\alpha^\pm$ such that 
\[
A(X)-\mathfrak A_\alpha^\pm(X)=B(Y)-\mathfrak B_\alpha^\pm(Y)=C_\alpha^\pm\quad\text{ for all }(X,Y)\in U^\pm_\alpha,
\]
so that
\[
\log_\pm R(X,Y)=A(X)+B(Y)-C_\alpha^\pm\quad\text{ for all }(X,Y)\in U^\pm_\alpha.
\]
Hence
\[
R(X,Y)=\exp(-C_\alpha^\pm)e^{A(X)}e^{B(Y)}\quad\text{ for all }(X,Y)\in U^\pm_\alpha.
\]
Thus
\[
R(X,Y)e^{-A(X)}e^{-B(Y)}=\exp(-C_\alpha^\pm)\quad\text{ for all }(X,Y)\in U^\pm_\alpha,
\]
and the right-hand side of this equality is locally constant on $\mathbb R^d\times\mathbb R^d$ which is connected, while the left-hand side is continuous, since $A$ and $B$ are continuous on $\mathbb R^d$. 
Hence both sides of this equality are equal to a constant on $\mathbb R^d\times\mathbb R^d$:
\[
R(X,Y)=re^{A(X)}e^{B(Y)}\quad\text{ for all }(X,Y)\in\mathbb R^d\times\mathbb R^d,
\]
where $r\not=0$ is a constant on $\mathbb R^d\times\mathbb R^d$. In particular 
\[
|r|\|e^A\|_{L^2(\mathbb R^d)}\|e^B\|_{L^2(\mathbb R^d)}=\|R\|_{L^2(\mathbb R^d\times\mathbb R^d)}\le\|R\|_{\mathcal L^1(L^2(\mathbb R^d))},
\]
so that $e^A\in L^2(\mathbb R^d)$ and $\mathrm{Ran}(R)=\mathbb C e^A$, which implies that $R$ is a rank-one operator.
\end{proof}

\begin{remark}
In the case of space dimension $d=1$, the identities in Lemma \ref{L-CharRk1} boil down to the single equality
\[
\tfrac4{\hbar^2}\partial_y\left(\frac{\partial_y\mathcal F_{\xi\to y}W_\hbar[R](x,y)}{\mathcal F_{\xi\to y}W_\hbar[R](x,y)}\right)
=\partial_x\left(\frac{\partial_x\mathcal F_{\xi\to y}W_\hbar[R](x,y)}{\mathcal F_{\xi\to y}W_\hbar[R](x,y)}\right),
\]
or, equivalently, to the d'Alembert equation (wave equation in space dimension $1$)
\[
\tfrac4{\hbar^2}\partial_y^2\log_\pm\tilde R(x,y)=\partial_x^2\log_\pm\tilde R(x,y)\qquad\text{ in }\mathcal D'(V_\pm),
\]
with speed of propagation $2/\hbar$. We recall that $V_\pm$ is the inverse image of $U_\pm$ by the linear transformation $(x,y)\mapsto(x+\tfrac\hbar{2}y,x-\tfrac\hbar{2}y)$, while $U_\pm=R^{-1}(\omega_\pm)$,
with $\omega_+:=\mathbb C\setminus[0,+\infty)$ and $\omega_-:=\mathbb C\setminus(-\infty,0]$.
\end{remark}

\smallskip
We shall later need the following implication of Lemma \ref{L-CharRk1}: if $R=R^*\in\mathcal L^1(L^2(\mathbb R^d))$ is a rank-one density operator satisfying the assumptions of Lemma \ref{L-CharRk1},
\[
\tilde R(x,-y)=\mathcal F_{\xi\to y}W_\hbar[R](x,y)
\] 
satisfies, for all $j,k=1,\ldots,d$, the identity
\begin{equation}\label{IdentPure}
\begin{aligned}
\tilde R(x,y)\partial_{y_j}\partial_{y_k}\tilde R(x,y)-&(\partial_{y_j}\tilde R(x,y))(\partial_{y_k}\tilde R(x,y))
\\
=&\tilde R(x,y)^2\partial_{y_j}\left(\frac{\partial_{y_k}\tilde R(x,y)}{\tilde R(x,y)}\right)=\tfrac14\hbar^2\tilde R(x,y)^2\partial_{x_j}\left(\frac{\partial_{x_k}\tilde R(x,y)}{\tilde R(x,y)}\right)
\\
=&\tfrac14\hbar^2\left(\tilde R(x,y)\partial_{x_j}\partial_{x_k}\tilde R(x,y)-(\partial_{x_j}\tilde R(x,y))(\partial_{x_k}\tilde R(x,y))\right)\,.
\end{aligned}
\end{equation}

\smallskip
However, since both sides of the second equality above, which follows from Lemma \ref{L-CharRk1}, are multiplied by $\tilde R(x,y)$, one can hope to avoid the assumption that $\tilde R(x,y)\not=0$ required
to apply Lemma \ref{L-CharRk1}. 

That this is indeed the case follows from the observation below, stated in an abstract setting.

\begin{lemma}\label{L-ImpliRk1} Let $D,D'$ be commuting derivations on a commutative algebra $\mathcal A$. Set
\[
T_{D,D'}(f):=DD'(f^2)-4(Df)(D'f)\,.
\]
For all $f,g\in\mathcal A$, one has
\[
T_{D,D'}(fg)=T_{D,D'}(f)g^2+f^2T_{D,D'}(g)\,.
\]
\end{lemma}

\begin{proof} By the Leibniz formula, one has indeed
\[
\begin{aligned}
T_{D,D'}(fg)=&D(g^2D'(f^2)+f^2D'(g^2))-4(gDf+fDg)(gD'f+fD'g)
\\
=&g^2DD'(f^2)+f^2DD'(g^2)+D(g^2)D'(f^2)+D(f^2)D'(g^2)
\\
&-4(gDf+fDg)(gD'f+fD'g)
\\
=&g^2DD'(f^2)+f^2DD'(g^2)+4gf(Dg)(D'f)+4fg(Df)(D'g)
\\
&-4g^2(Df)(D'f)-4f^2(Dg)(D'g)-4fg(Dg)(D'f)-4gf(Df)D'g)
\\
=&g^2DD'(f^2)+f^2DD'(g^2)-4g^2(Df)(D'f)-4f^2(Dg)(D'g)
\\
=&T_{D,D'}(f)g^2+f^2T_{D,D'}(g)\;.
\end{aligned}
\]
\end{proof}

\smallskip
With this, we shall establish the equality between the leftmost and rightmost sides of \eqref{IdentPure} \textit{without} using the nonvanishing condition $\mathcal F_{\xi\to y}W_\hbar[R]\not=0$ everywhere 
on $\mathbb R^d\times\mathbb R^d$ assumed in Lemma \ref{L-CharRk1}. This equality is at the core of our discussion of semiclassical velocity averaging for pure states in the next section.

\begin{corollary}\label{C-ImpliRk1}
Consider the pure state density operator $R:=|\psi\rangle\langle\psi|$, where $\psi\in H^2(\mathbb R^d)$. Then 
\[
\tilde R(x,y):=\mathcal F_{\xi\to y}W_\hbar[R](x,-y)
\]
satisfies, for all $j,k=1,\ldots,d$, the identity
\begin{equation}\label{IdentBisPure}
\begin{aligned}
\tilde R(x,y)\partial_{y_j}\partial_{y_k}\tilde R(x,y)-&(\partial_{y_j}\tilde R(x,y))(\partial_{y_k}\tilde R(x,y))
\\
&=\tfrac14\hbar^2\left(\tilde R(x,y)\partial_{x_j}\partial_{x_k}\tilde R(x,y)-(\partial_{x_j}\tilde R(x,y))(\partial_{x_k}\tilde R(x,y))\right)\,.
\end{aligned}
\end{equation}
\end{corollary}

\begin{proof}
Let us first treat the case where $\psi$ is smooth. With the notations of Lemma \ref{L-ImpliRk1}, set
\[
T_{x_j,x_k}:=T_{\partial_{x_j},\partial_{x_k}}\,,\qquad T_{y_j,y_k}:=T_{\partial_{y_j},\partial_{y_k}}\,,\qquad j,k=1,\ldots,d\,.
\]
Obviously
\[
\partial_{y_\ell}\psi(x\pm\tfrac12\hbar y)=\pm\tfrac12\hbar\partial_{x_\ell}\psi(x\pm\tfrac12\hbar y)\,,
\]
so that
\[
T_{y_j,y_k}(\psi(x\pm\tfrac12\hbar y))=\tfrac14\hbar^2T_{x_j,x_k}(\psi(x\pm\tfrac12\hbar y))\,,\qquad j,k=1,\ldots,d\,.
\]
Applying Lemma \ref{L-ImpliRk1} with $f=\psi(x\pm\tfrac12\hbar y)$ and $g=\overline{\psi(x-\tfrac12\hbar y)}$ while $D=\partial_{y_j}$ or $\partial_{x_j}$ and $D'=\partial_{y_k}$ or $\partial_{x_k}$ shows that
\[
\tilde R(x,y):=\psi(x+\tfrac12\hbar y)\overline{\psi(x-\tfrac12\hbar y)}
\]
satisfies
\[
T_{y_j,y_k}(\tilde R(x,y))=\tfrac14\hbar^2T_{x_j,x_k}(\tilde R(x,y))\,,\qquad j,k=1,\ldots,d\,.
\]
The conclusion follows from observing that
\[
\begin{aligned}
T_{y_j,y_k}(\tilde R(x,y))=&\partial_{y_j}\partial_{y_k}(\tilde R(x,y)^2)-4(\partial_{y_j}\tilde R(x,y))(\partial_{y_k}\tilde R(x,y))
\\
=&2\partial_{y_j}(\tilde R(x,y)\partial_{y_k}\tilde R(x,y))-4(\partial_{y_j}\tilde R(x,y))(\partial_{y_k}\tilde R(x,y))
\\
=&2\tilde R(x,y)\partial_{y_j}\partial_{y_k}\tilde R(x,y)-2(\partial_{y_j}\tilde R(x,y))(\partial_{y_k}\tilde R(x,y))\,,
\end{aligned}
\]
and similarly
\[
T_{x_j,x_k}(\tilde R(x,y))=2\tilde R(x,y)\partial_{x_j}\partial_{x_k}\tilde R(x,y)-2(\partial_{x_j}\tilde R(x,y))(\partial_{x_k}\tilde R(x,y))\,.
\]

If $\psi\in H^2(\mathbb R^d)$, then $(X,Y)\mapsto\psi(X)\overline{\psi(Y)}$ belongs to $H^2(\mathbb R^d\times\mathbb R^d)$. This is most easily seen in Fourier variables, since
\[
(1+|\xi|+|\eta|)^2|\hat\psi(\xi)||\hat\psi(\eta)|\le(1+|\xi|)^2|\hat\psi(\xi)|(1+|\eta|)^2|\hat\psi(\eta)|\in L^2(\mathbb R^d_\xi\times\mathbb R^d_\eta)\,.
\]
Hence $\tilde R\in H^2(\mathbb R^d\times\mathbb R^d)$ (as the composition of the function $(X,Y)\mapsto\psi(X)\overline{\psi(Y)}$ which belongs to $H^2(\mathbb R^d\times\mathbb R^d)$  with the linear
diffeomorphism  $(x,y)\mapsto(x+\tfrac\hbar{2}y,x-\tfrac\hbar{2}y)$). Therefore $\tilde R$, $\partial_{y_\ell}\tilde R$ and $\partial_{y_j}\partial_{y_k}\tilde R$ belong to $L^2(\mathbb R^d\times\mathbb R^d)$,
and the same is true of $\partial_{x_\ell}\tilde R$ and $\partial_{x_j}\partial_{x_k}\tilde R$. Hence all the products involved in the identity \eqref{IdentBisPure} belong to $L^1(\mathbb R^d\times\mathbb R^d)$,
and \eqref{IdentBisPure} is an equality a.e. between functions in $L^1(\mathbb R^d\times\mathbb R^d)$.

Since we have seen that \eqref{IdentBisPure} holds true if $\psi$ is smooth, the case of $\psi\in H^2(\mathbb R^d)$ follows by density of the Schwartz space $\mathcal S(\mathbb R^d)$ of smooth functions
with rapidly decaying derivatives of all orders in $H^2(\mathbb R^d)$.
\end{proof}

\begin{remark}
Notice that \eqref{IdentBisPure} is a \textit{consequence} of the assumption that $R$ is a pure state density operator, which does \textit{not} require the condition $\psi\not=0$ used in Lemma \ref{L-CharRk1}. 
As we explained in \eqref{IdentPure}, this identity is a straightforward consequence of Lemma \ref{L-CharRk1} if one knows that $\psi\not=0$. We do not claim that \eqref{IdentBisPure} can be used to obtain
a characterization of pure states (i.e. the converse of that implication) without assuming $\psi\not=0$ as in Lemma \ref{L-CharRk1}. Although this converse implication is not used in this paper, we think that 
Lemma \ref{L-CharRk1} is a result of independent interest as it (a) makes Tatarski\u\i's formal argument in \cite{tatarskiui1983wigner} fully rigorous and (b) extends it to space dimensions higher than one.
\end{remark}


\section{Semiclassical velocity averaging and pure states}\label{sec:PureSt}


With Corollary \ref{C-ImpliRk1}, we shall now discuss the possibility of obtaining a semiclassical velocity averaging result similar to Theorem \ref{thm:VAWignerL2}  in the case of a family of pure quantum states.

\begin{proposition}\label{P-Monokinetic}
Let $\psi_\hbar\in H^2(\mathbb R^d)$ be a family of wave functions such that $\|\psi_\hbar\|_{L^2(\mathbb R^d)}=1$, and let $w_\hbar:=W_\hbar[\ketbra{\psi_\hbar}]$. Assume that a subsequence $\psi_{\hbar_n}$
satisfies
\[
w_{\hbar_n}\to w\text{ in }\mathcal S'(\mathbb R^d\times\mathbb R^d),\quad\text{ and }\quad\hbar_n\|\nabla\rho_{\hbar_n}\|_{L^2(B(0,R))}\to 0
\]
for all $R>0$ as $\hbar_n\to 0$. Assume that
\begin{equation}\label{Def+LimE}
\mathcal E_{\hbar_n}=\tfrac1{2m}\int_{\mathbb R^d}|\xi|^2w_{\hbar_n}\dd\xi\to\mathcal E:=\tfrac1{2m}\int_{\mathbb R^d}|\xi|^2w\dd\xi\quad\text{ weakly in }L^2(\mathbb R^d)
\end{equation}
as $\hbar_n\to 0$, while
\begin{equation}\label{Def+LimRhoJ}
\rho_{\hbar_n}:=\int_{\mathbb R^d}w_{\hbar_n}\dd\xi\to\rho:=\int_{\mathbb R^d}w\dd\xi\quad\text{ and }J_{\hbar_n}:=\tfrac1m\int_{\mathbb R^d}\xi w_{\hbar_n}\dd\xi\to J:=\tfrac1m\int_{\mathbb R^d}\xi w\dd\xi
\end{equation}
in $L^2_{loc}(\mathbb R^d)$ as $\hbar_n\to 0$ for each $R>0$. Set
\[
u(x):=\frac{\mathbf 1_{\rho(x)>0}}{\rho(x)}J(x)\in\mathbb R^d.
\]
Then $w$ is a monokinetic, positive Borel measure on the phase space $\mathbb R^d\times\mathbb R^d$, i.e.
\[
w=\rho(x)\delta(\xi-u(x)).
\]
\end{proposition}

That $w_{\hbar_n}$ converges to a monokinetic Wigner measure $w=\rho(x)\delta(\xi-u(x))$ is a very strong indication that the strong convergence of the first moments $\rho_{\hbar_n}$ and $J_{\hbar_n}$ 
in $L^2_{loc}(\mathbb R^d)$ cannot be deduced from a velocity averaging lemma. Indeed, velocity averaging is based on the fact that, at the kinetic level of description, the orthogonal projections of the variable
$\xi$ on each line through the origin are regularly distributed: see condition (2.1) in \cite{golse1988regularity}. This is incompatible with a situation where the $\xi$ variable concentrates on a single value at
each position $x$, which is precisely the case of a monokinetic distribution function.

Thus, the conclusions expected from a velocity averaging lemma in the present setting, namely the strong convergence of the moments of order $0$ and $1$, i.e. of $\rho_{\hbar_n}$ and $J_{\hbar_n}$, lead to 
a conclusion \textit{which precludes using velocity averaging}! 

Put in other words, this observation does obviously not prevent $\rho_\hbar$ and $J_\hbar$ to be strongly relatively compact families of $L^2_{loc}$, but in that case, the strong compactness of these families 
cannot be expected to follow from velocity averaging.

\begin{remark}
The impossibility to deduce the strong compactness in $L^2_{loc}$ of the density functions and momentum densities of a family of pure states on $L^2$ from a velocity averaging lemma is based on the observation that this strong compactness
implies that Wigner measures of such states must be monokinetic. Conditions under which the Wigner measures of a family of pure states are monokinetic have been studied in detail in section 4.3 of \cite{MarkoPS}: see Theorem 4.5 and
Corollary 4.7 there. However, this reference assumes that $\psi_\eps$ is of the form
\[
\psi_\eps(x)=\sqrt{\rho_\eps(x)}e^{iS_\eps(x)/\eps}
\]
and that $\nabla S_\eps$ converges uniformly on some open neighborhood of the support of $\rho$, the limit of $\rho_\eps$ in $L^1$ to a $C^1$ function, or on some open neighborhood of the union of the supports of all the $\rho_\eps$. 
These assumptions are \textit{much stronger} than what can be deduced from velocity averaging. (Indeed, velocity averaging implies at best a gain of a $1/2$ derivative in $L^2_{loc}$, and never $C^1$ regularity.) And the difference between
the assumptions needed in section 4.3 of \cite{MarkoPS} and our Proposition \ref{P-Monokinetic} are not only minor technicalities. Indeed, it is \textit{absolutely essential} for our purpose in the present paper that the assumptions implying 
that the Wigner measure of a sequence is monokinetic coincide exactly with the compactness conclusions of velocity averaging in order to arrive at a contradiction. 
This is why we cannot appeal to  the results from section 4.3 in \cite{MarkoPS} to reach this contradiction.
\end{remark}

\begin{proof}
Call $\tilde R_\hbar(x,y):=\psi(x+\tfrac{\hbar}2y)\overline{\psi(x-\tfrac{\hbar}2y)}$. Observe that
\begin{equation}\label{FlaRhoJE}
\rho_{\hbar_n}(x)=\tilde R_{\hbar_n}(x,0)\,,\quad J_{\hbar_n}(x)=\tfrac1m(-i\nabla_y\tilde R_{\hbar_n}(x,0))\,,\quad\mathcal E_{\hbar_n}(x)=\tfrac1{2m}(-\Delta_y\tilde R_{\hbar_n}(x,0)).
\end{equation}
Next we use the identity in Corollary \ref{C-ImpliRk1} for $j=k$ and sum over $j=1,\ldots,d$, to find that
\[
\rho_{\hbar_n}(x)\Delta_y\tilde R_{\hbar_n}(x,0)=|\nabla_y\tilde R_{\hbar_n}(x,0)|^2+\tfrac{\hbar_n^2}4\rho_{\hbar_n}(x)^2\Delta_x\ln\rho_{\hbar_n}(x),
\]
or, equivalently
\[
2m\rho_{\hbar_n}(x)\mathcal E_{\hbar_n}(x)-m^2|J_{\hbar_n}(x)|^2=-\tfrac{\hbar_n^2}8\Delta_x\left(\rho_{\hbar_n}(x)^2\right)+\tfrac{\hbar_n^2}2|\nabla_x\rho_{\hbar_n}(x)|^2.
\]
Since $\rho_{\hbar_n}\to\rho$ and $J_{\hbar_n}\to J$ in $L^2(B(0,R))$ while $\mathcal E_{\hbar_n}\to\mathcal E$ weakly in $L^2(\mathbb R^d)$, one has
\[
\rho_{\hbar_n}(x)\mathcal E_{\hbar_n}\to\rho\mathcal E\text{ in }\mathcal D'(\mathbb R^d)\quad\text{ while }|J_{\hbar_n}|^2\to|J|^2\text{ in }L^1(B(0,R)),
\]
and
\[
\hbar_n^2\rho_{\hbar_n}^2\to 0\text{ in }L^1(B(0,R)),\quad\text{ so that }\hbar_n^2\Delta_x\left(\rho_{\hbar_n}^2\right)\to 0\text{ in }\mathcal D'(\mathbb R^d).
\]
Since $\hbar_n\|\nabla\rho_{\hbar_n}\|_{L^2(B(0,R))}\to 0$, passing to the limit in both sides of the equality above as $\hbar_n\to 0$ implies that
\[
2m\rho\mathcal E-m^2|J|^2=\int_{\mathbb R^d}w\dd\xi\int_{\mathbb R^d}|\xi|^2w\dd\xi-\left|\int_{\mathbb R^d}\xi w\dd\xi\right|^2=0\qquad\text{ a.e. on }\mathbb R^d.
\]
Multiplying both sides of this identity by $\mathbf 1_{\rho(x)>0}/\rho(x)^2$ shows that
\[
\frac{\mathbf 1_{\rho>0}}{\rho}\int_{\mathbb R^d}\left|\xi-\frac{\mathbf 1_{\rho>0}}{\rho}\int_{\mathbb R^d}\zeta w\dd\zeta\right|^2w\dd\xi
=\frac{\mathbf 1_{\rho>0}}{\rho}\int_{\mathbb R^d}|\xi|^2w\dd\xi-\left|\frac{\mathbf 1_{\rho>0}}{\rho}\int_{\mathbb R^d}\xi w\dd\xi\right|^2=0.
\]
so that, since $w$ is a positive (Borel) measure on $\mathbb R^d\times\mathbb R^d$,
\[
w=\rho(x)\delta(\xi-u(x)),\quad\text{ with }u:=\frac{\mathbf 1_{\rho>0}}{\rho}\int_{\mathbb R^d}\zeta w\dd\zeta.
\]
\end{proof}

\begin{remark}
The semiclassical regime is an essential feature in the negative result reported in Proposition \ref{P-Monokinetic}. In situations other than the semiclassical regimes, velocity averaging can be used even
in the case of pure states as in Theorem \ref{thm:VAWignerL2}. Returning to the discussion preceding Theorem \ref{thm:VAWignerL2}, and assuming that $\hbar\ge\hbar_0>0$, the inequality
\[
\sup_{\hbar_0\le\hbar\le 1}\|W_\hbar[R_\hbar]\|_{L^2}^2=C<\infty
\]
is obviously compatible with the constraint 
\[
1\le\text{rank}R_\hbar\cdot(2\pi\hbar)^d\sup_{\hbar_0\le\hbar\le 1}\|W_\hbar[R_\hbar]\|_{L^2}^2
\]
in the case $\text{rank}R_\hbar=1$, as can be seen by taking 
\[
C>\frac1{(2\pi\hbar_0)^d}\,.
\]
In fact, the idea of applying velocity averaging to the Wigner equation away from the semiclassical regime and in the case of pure states goes back to the work of Lions and Perthame \cite{lions1992lemmes}.
This reference explains the connection between dispersion for the free Schr\"odinger equation and velocity averaging for the Wigner equation --- see also \cite{GasserMarkoPerth} for an extension of
these results to the case of the Schr\"odinger equation with an external potential. However, this last reference is not based on velocity averaging itself, but on dispersion effects on the Wigner equation
formulated in terms of moment estimates (another issue discussed in \cite{lions1992lemmes}).
\end{remark}

\begin{remark}
Let us recall Theorem III.1 (5) in \cite{lions1993mesures}: given any Borel probability measure $\mu$ on $\mathbb R^d\times\mathbb R^d$, there exists a family $\psi_\hbar$ of $L^2(\mathbb R^d)$ such that
$\|\psi_\hbar\|_{L^2(\mathbb R^d)}=1$ and $\psi_\hbar\to 0$ weakly in $L^2(\mathbb R^d)$ while $W_\hbar[\ketbra{\psi_\hbar}]\to\mu$ in $\mathcal S'(\mathbb R^d\times\mathbb R^d)$ as $\hbar\to 0$.
Since $\mu$ is not necessarily monokinetic, this means that the additional assumptions in Proposition \ref{P-Monokinetic} are in general not satisfied by the family $\psi_\hbar$, and that these additional
assumptions are essential for Proposition \ref{P-Monokinetic} to hold.
\end{remark}

\begin{remark}\label{R-CondRho}
The condition on $\nabla\rho_\hbar$ in Proposition \ref{P-Monokinetic} is obviously an essential feature of this result. We shall return to its physical meaning at the end of this paper. On the other hand,
assume that the kinetic energy corresponding to the pure state defined by the wave function $\psi_\hbar\in L^2(\mathbb R^d)$ such that $\|\psi_\hbar\|_{L^2(\mathbb R^d)}=1$ is uniformly bounded in $\hbar$:
\[
(\psi_\hbar|-\tfrac{\hbar^2}{2m}\Delta\psi_\hbar)_{L^2(\mathbb R^d)}=\tfrac{\hbar^2}{2m}\|\nabla\psi_\hbar\|^2_{L^2(\mathbb R^d)}
=\tfrac1{(2\pi)^d}\int_{\mathbb R^d}\tfrac{\hbar^2}{2m}|\xi|^2|\mathcal F_{x\to\xi}\psi_\hbar(\xi)|^2\dd \xi\le C.
\]
Then
\[
\sup_{\hbar>0}\int_{|\xi|>R/\hbar}|\mathcal F_{x\to\xi}\psi_\hbar(\xi)|^2d\xi\le(2\pi)^d\frac{2mC}{R^2}\to 0\quad\text{ as }R\to+\infty.
\]
See (1.28) and (1.26) in \cite{gerard1997homogenization} for a more general sufficient condition leading to that conclusion than the boundedness of the kinetic energy. In other words the family $\psi_\hbar$ 
is \emph{$\hbar$-oscillatory} in the terminology of Definition 1.6 in \cite{gerard1997homogenization}. Since $\widehat{\rho}_{\hbar}(\xi)=|\mathcal F_{x\to\xi}\psi_\hbar(\xi)|^2$ (see \eqref{DefHatRho} for the
definition of $\widehat{\rho}_\hbar$), one has
\[
\int_{|\zeta|>R}\frac1{\hbar^d}\widehat{\rho}_\hbar\left(\frac\zeta\hbar\right)d\zeta=\int_{|\xi|>R/\hbar}|\mathcal F_{x\to\xi}\psi_\hbar(\xi)|^2d\xi.
\]
Thus, the tightness of $\hbar^{-d}\widehat{\rho}_{\hbar}(\xi/\hbar)$ is equivalent to $\psi_\hbar$ being $\hbar$-oscillatory.
\end{remark}

\smallskip
Here are some examples of families $\psi_\hbar$ satisfying the condition $\hbar\|\nabla\rho_\hbar\|_{L^2(B(0,R))}\to 0$ as $\hbar\to 0^+$.

\begin{example} (WKB states) Set
\[
\psi_\hbar(x):=a_\hbar(x)e^{iS_\hbar(x)/\hbar},
\]
with $a_\hbar$ and $S_\hbar$ of class $C^1$ and real-valued on $\mathbb R^d$. Then $\rho_\hbar(x)=a_\hbar(x)^2$ and
\[
\hbar^2\|\nabla\rho_\hbar(x)\|^2_{L^2(B(0,R))}=4\hbar^2\int_{B(0,R)}a_\hbar(x)^2|\nabla a_\hbar(x)|^2\dd x.
\]
This is to be compared with 
\[
\begin{aligned}
\hbar^2\|\nabla\psi_\hbar\|^2_{L^2(B(0,R))}=&\int_{B(0,R)}|\hbar\nabla a_\hbar(x)+ia_\hbar(x)\nabla S_\hbar(x)|^2\dd x
\\
=&\int_{B(0,R)}(\hbar^2|\nabla a_\hbar(x)|+a_\hbar(x)^2|\nabla S_\hbar(x)|^2)\dd x.
\end{aligned}
\]
\end{example}

\begin{example}(Coherent states)
Set 
\[
\psi_{q,p}(x):=(\pi\hbar)^{-d/4}e^{-|x-q|^2/2\hbar}e^{ip\,\cdot(x-q/2)/\hbar}
\]
for $q,p\in\mathbb R^d$. Then
\[
\rho_\hbar(x)=(\pi\hbar)^{-d/2}e^{-|x-q|^2/\hbar}
\]
so that
\[
\hbar\nabla\rho_\hbar(x)=2(x-q)\rho_\hbar(x)
\]
and
\[
\begin{aligned}
\hbar^2\|\nabla\rho_\hbar\|^2_{L^2(\mathbb R^d)}=&4\int_{\mathbb R^d}|z|^2(\pi\hbar)^de^{-2|z|^2/\hbar}\dd z
\\
=&\hbar^{1-\frac{d}2}\int_{\mathbb R^d}\tfrac{4}{\pi^d}|y|^2e^{-2|y|^2}\dd y=\tfrac{d}{(4\pi)^{d/2}}\hbar^{1-\frac{d}2}\to 0
\end{aligned}
\]
if and only if $d=1$.
\end{example}

\begin{example}
Here is another class of examples, very close to the case of coherent states:
\[
\psi_\alpha(x):=\hbar^{-d\alpha/2}a(x/\hbar^\alpha)e^{ip\,\cdot x/\hbar},
\]
assuming that
\[
\|\psi\|^2_{L^2(\mathbb R^d)}=\hbar^{-d\alpha}\int_{\mathbb R^d}a(x/\hbar^\alpha)^2dx=\int_{\mathbb R^d}a(y)^2\dd y=1.
\]
Then
\[
\rho_\hbar(x)=\hbar^{-d\alpha}a(x/\hbar^\alpha)^2\,,\quad\nabla\rho_\hbar(x)=2\hbar^{-(d+1)\alpha}a(x/\hbar^\alpha)\nabla a(x/\hbar^\alpha)
\]
and
\[
\begin{aligned}
\hbar^2\|\nabla\rho_\hbar\|^2_{L^2(\mathbb R^d)}=&4\hbar^2\hbar^{-2(d+1)\alpha}\int_{\mathbb R^d}a(x/\hbar^\alpha)^2|\nabla a(x/\hbar^\alpha)|^2\dd x
\\
=&4\hbar^2\hbar^{-2(d+1)\alpha}\hbar^{d\alpha}\int_{\mathbb R^d}a(y)^2|\nabla a(y)|^2\dd y
\\
=&\hbar^{2-(d+2)\alpha}\|\nabla(a^2)\|^2_{L^2(\mathbb R^d)}\to 0
\end{aligned}
\]
if and only if $\alpha<\frac2{d+2}$ provided that $a^2\in H^1(\mathbb R^d)$.
\end{example}


\section{Using Corollary \ref{C-ImpliRk1} to derive Madelung's Quantum Hydrodynamic Equations}\label{sec:DerMadelung}


Here is another application of Corollary \ref{C-ImpliRk1} leading to a  a quick and very natural derivation of Madelung's quantum hydrodynamic equations. These equations were obtained for the first time in 
\cite{Madelung1926} (see especially equations (4') and (3'') in that reference). In Madelung's own words, his work \cite{Madelung1926} corresponds mostly to an analogy between Schr\"odinger's approach 
to quantum dynamics with hydrodynamics (see \cite{Madelung1926} on p. 322). 

The presentation of Madelung's hydrodynamic equations below is closer in spirit to \cite{GasserMarko}, with however significant differences --- see the remark following Theorem \ref{T-Madelung} below.
Earlier discussions of Madelung's quantum hydrodynamics can be found in \cite{PGX,AlazardCarles,Piero1,Piero2}, and in many other references.

Although Madelung's system of hydrodynamic equations is essentially independent of the question of velocity averaging for the Wigner equation at first sight, we shall see that the condition used in Proposition 
\ref{P-Monokinetic}
\[
\hbar_n\|\nabla\rho_{\hbar_n}\|_{L^2(B(0,R))}\to 0,
\]
already discussed in Remark \ref{R-CondRho}, but whose physical meaning remains somewhat mysterious, has a very natural interpretation in terms of the Bohm potential that appears in Madelung's system.

Let $\psi\equiv\psi(t,x)\in\mathbb C$ be a solution of the Schr\"odinger equation
\begin{equation}\label{SchroV}
	i\hbar\partial_t\psi(t,x)=-\tfrac{\hbar^2}{2m}\Delta_x\psi(t,x)+V(x)\psi(t,x).
\end{equation}
The von Neumann equation \eqref{eq:von Neumann R tilde} for $R(t)=\ketbra{\psi(t,\cdot)}$ is recast as
\begin{equation}
    \partial_t\tilde{R}(t,x,y)+\tfrac1m\sum_{j=1}^d \partial_{x_j}(-i\partial_{y_j})\tilde{R}(t,x,y)=\frac{V(x+\frac{\hbar}{2}y)-V(x-\frac{\hbar}{2}y)}{i\hbar}\tilde{R}(t,x,y),
    \label{eq:von Neumann Madelung}
\end{equation}
where we recall that
\begin{equation}\label{R=FW}
\tilde R(t,x,y)=\psi(t,x+\tfrac\hbar{2}y)\overline{\psi(t,x-\tfrac\hbar{2}y)}=\mathcal F_{\xi\to y}W_\hbar[R](t,x,-y).
\end{equation}

Let us first discuss the case where $\psi$ is smooth and $\psi(t,x)\not=0$ for all $t$ and all $x$. Evaluating both sides of \eqref{eq:von Neumann Madelung} at $y=0$, we find  that
\begin{equation}\label{vNy=0}
    \partial_t\tilde{R}(t,x,0)+\tfrac1m\sum_{j=1}^d\partial_{x_j}(-i\partial_{y_j}\tilde{R}(t,x,0))=0. 
\end{equation}
Define the density function $\rho(t,x)$ as in \eqref{DefDensityFunc}, and the velocity field $u(t,x)\in\mathbb R^d$ by
\begin{equation}\label{RhoU}
\rho(t,x):=\tilde{R}(t,x,0)=|\psi(t,x)|^2>0,\quad\text{ and }\quad u(t,x):=\frac{-i\nabla_y\tilde R(t,x,0)}{m\tilde R(t,x,0)}\,.
\end{equation}
With these definitions, \eqref{vNy=0} takes the form of the continuity equation
\begin{equation}
    \partial_t\rho(t,x)+\nabla_x\cdot(\rho u)(t,x)=0.
    \label{eq:continuity}
\end{equation}
For the Euler (momentum) equation, we apply $-\tfrac{i}m\partial_{y_k}$ to both sides of \eqref{eq:von Neumann Madelung}, and arrive at
\begin{equation}
    \partial_t(-\tfrac{i}m\partial_{y_k}\tilde{R}(t,x,y))\!-\!\tfrac1{m^2}\sum_{j=1}^d\partial_{x_j}(\partial_{y_j}\partial_{y_k} \tilde{R}(t,x,y))\!=\!-\tfrac1m\partial_{y_k}\left(\!\tfrac{V(x+\frac{\hbar}{2}y)-V(x-\frac{\hbar}{2}y)}{\hbar}\tilde{R}(t,x,y)\!\right)\,.
    \label{eq:euler1}
\end{equation}
Specializing this equality to $y=0$, and using the quantities defined in \eqref{RhoU} shows that 
\begin{equation}\label{eq:euler2}
    \partial_t(\rho u_k)(t,x)-\tfrac1{m^2}\sum_{j=1}^d\partial_{x_j}(\partial_{y_j}\partial_{y_k}\tilde{R}(t,x,0)) =-\tfrac1m\partial_kV(x)\rho(t,x)\,.
\end{equation}
The computations leading to \eqref{eq:continuity} and \eqref{eq:euler2} is classical and have been recalled here for the reader's convenience. It remains to find a closure relation expressing $\partial_{y_j}\partial_{y_k}\tilde{R}(t,x,0)$ in terms 
of $\rho$ and $u$. This step is quite straightforward in Madelung's original approach \cite{Madelung1926}, but much less so in the Gasser-Markowich approach.

This task is greatly simplified by the identity \eqref{IdentPure} following from Lemma \ref{L-CharRk1}, recast as
\begin{equation}\label{PrePreClo}
\begin{aligned}
-\tilde R(t,x,y)\partial_{y_j}\partial_{y_k}\tilde{R}(t,x,y)-(-i\partial_{y_j}\tilde{R}(t,x,y))(-i\partial_{y_k}\tilde{R}(t,x,y))
\\
=-\tfrac14\hbar^2\left(\tilde R(t,x,y)\partial_{x_j}\partial_{x_k}\tilde{R}(t,x,y)-(\partial_{x_j}\tilde{R}(t,x,y))(\partial_{x_k}\tilde{R}(t,x,y))\right)&\,,
\end{aligned}
\end{equation}
which is evaluated to $y=0$, to find
\begin{equation}\label{PreClo}
\begin{aligned}
-\rho(t,x)\partial_{y_j}\partial_{y_k}\tilde{R}(t,x,0)-m^2\rho(t,x)^2u_j(t,x)u_k(t,x)
\\
=-\tfrac14\hbar^2\left(\rho(t,x)\partial_{x_j}\partial_{x_k}\rho(t,x)-(\partial_{x_j}\rho(t,x))(\partial_{x_k}\rho(t,x))\right)
\\
=-\tfrac14\hbar^2\rho(t,x)\left(\partial_{x_j}\partial_{x_k}\rho(t,x)-4\left(\partial_{x_j}\sqrt{\rho(t,x)}\right)\left(\partial_{x_k}\sqrt{\rho(t,x)}\right)\right)&\,.
\end{aligned}
\end{equation}
Dividing the leftmost and rightmost sides of \eqref{PreClo} by $\rho>0$ leads to the closure relation
\begin{equation}\label{Closure}
\begin{aligned}
-\partial_{y_j}\partial_{y_k}\tilde{R}(t,x,0)=m^2\rho(t,x)u_j(t,x)u_k(t,x)
\\
-\tfrac14\hbar^2\left(\partial_{x_j}\partial_{x_k}\rho(t,x)-4\left(\partial_{x_j}\sqrt{\rho(t,x)}\right)\left(\partial_{x_k}\sqrt{\rho(t,x)}\right)\right)&\,.
\end{aligned}
\end{equation}
We next insert this expression of $-\partial_{y_j}\partial_{y_k}\tilde{R}(t,x,0)$ in \eqref{eq:euler2} to find the momentum equation
\begin{equation}\label{eq:Madelung Euler}
\begin{aligned}
\partial_t(\rho u)(t,x)+&\nabla_x\cdot(\rho u^{\otimes 2})(t,x)=-\tfrac1m\rho\nabla V(x)
\\
+&\tfrac{\hbar^2}{4m^2}\nabla_x\cdot\left(\nabla_x^2\rho(t,x)-4\left(\nabla_x\sqrt{\rho(t,x)}\right)^{\otimes 2}\right)\,.
\end{aligned}
\end{equation}
Summarizing, we have explained how Lemma \ref{L-CharRk1} and its consequence the identity \eqref{IdentPure} lead to the closure relation needed in the derivation of Madelung's quantum hydrodynamical system \eqref{eq:continuity}
and \eqref{eq:Madelung Euler}. Unfortunately, the argument above uses the smoothness of $\psi$ and the fact that $\psi\not=0$ everywhere, which does not improve on \cite{Madelung1926}.

However, we shall see that the same approach as above, using Corollary \ref{C-ImpliRk1} instead of Lemma \ref{L-CharRk1}, allows us weaken the assumptions on regularity and vacuum ($\psi=0$).

\begin{theorem}\label{T-Madelung}
Let $\psi\in C([0,\tau);H^2(\mathbb R^d))$ be a solution of \eqref{SchroV} with $V\in W^{1,\infty}(\mathbb R^d)$. Assume that $\psi$ satisfies, for all $t>0$, the condition
\begin{equation}\label{psinot=0}
	\psi(t,x)\not=0,\qquad\text{ for a.e. }x\in\mathbb R^d.
\end{equation}
Then the density function $\rho:=|\psi|^2$ as in \eqref{DefDensityFunc}, and the velocity field defined by
\begin{equation}\label{DefVelField}
u(t,x):=\mathbf 1_{\rho(t,x)>0}\frac{-i\nabla_y\tilde R(t,x,0)}{m\tilde R(t,x,0)}
\end{equation}
satisfy Madelung's system of quantum hydrodynamics on $[0,\tau)\times\mathbb R^d$:
\begin{equation}\label{MadelungSyst}
\left\{
\begin{aligned}
	{}&\partial_t\rho+\nabla_x\cdot(\rho u)=0,
	\\
	&\partial_t(\rho u)+\nabla_x\cdot(\rho u^{\otimes 2})=\tfrac{\hbar^2}{4m^2}\nabla_x\!\cdot\big(\nabla^2_x\rho-4(\nabla_x\sqrt{\rho})^{\otimes 2}\big)-\tfrac1m\rho\nabla_xV.
\end{aligned}
\right.
\end{equation}
\end{theorem}

\begin{remark}\label{R-RegularityMadelung}
Our assumption that $\psi\in C([0,\tau);H^2(\mathbb R^d))$ is quite natural for solutions of the Cauchy problem for \eqref{SchroV}. Indeed, if $\psi(0,\cdot)$ is in the domain of $\mathcal H:=-\tfrac{\hbar^2}{2m}\Delta+V$, that is a 
self-adjoint operator on $L^2(\mathbb R^d)$, the solution $\psi(t,\cdot)=\exp(-it\mathcal H/\hbar)\psi(0,\cdot)$ of the Cauchy problem for \eqref{SchroV} is continuous and bounded in $t$ with values in the domain of $\mathcal H$ by
Stone's theorem. Besides, under the assumption that $V\in L^\infty(\mathbb R^d)$, the domain of $\mathcal H$ is $H^2(\mathbb R^d)$. In that case, our assumption that $\psi\in C([0,\tau);H^2(\mathbb R^d))$ is satisfied provided that
the initial data for \eqref{SchroV} belongs to $H^2(\mathbb R^d)$.
\end{remark}

\begin{proof}
First we recall that, if $\phi\in L^2(\mathbb R^d)$, the function $y\mapsto\phi(x+\tfrac\hbar{2}y)\overline{\phi(x\!-\!\tfrac\hbar{2}y)}$ belongs\footnote{Indeed, for all $u\in L^2(\mathbb R^d)$, one has
\[
\int_{\mathbb R^d}|u(x+z)-u(x)|^2dx\to 0\quad\text{ as }|z|\to 0\,.
\]
}
to $C_b(\mathbb R^d_y;L^1(\mathbb R^d_x))$. Applying this to the functions $\partial_\ell\psi(t,\cdot)$ and $\partial_j\partial_k\psi(t,\cdot)\in L^2(\mathbb R^d)$ justifies setting $y=0$ in \eqref{eq:von Neumann Madelung} and
in \eqref{eq:euler1} to arrive at \eqref{vNy=0} and \eqref{eq:euler2}, once these equalities are written in the sense of distributions in $(t,x)$, meaning that both sides of each equality are integrated against smooth, compactly 
supported test functions in $(t,x)$. 

Then, the equality \eqref{PreClo} is justified in the same manner from \eqref{PrePreClo}, which follows from Corollary \ref{C-ImpliRk1}, only assuming Sobolev $H^2$ regularity in the position variable.

It remains to justify \eqref{Closure}. Since we have assumed that $\psi\in C_b([0,\tau);H^2(\mathbb R^d_x))$, the functions $\partial_{y_j}\partial_{y_k}\tilde{R}\big|_{y=0}$ and $\partial_{x_j}\partial_{x_k}\rho$ both belong to 
$C_b([0,\tau);L^1(\mathbb R^d))$. Next we use the Sobolev embedding theorem to see that 
\[
\psi\text{ and }\partial_{x_\ell}\psi\in C_b(\mathbb R_t;H^{1/4}(\mathbb R^d))\subset C_b([0,\tau);L^p(\mathbb R^d))\,,\qquad p:=\tfrac{4d}{2d-1}\,,
\]
so that
\[
\partial_{x_j}\partial_{x_k}\rho=\overline{\psi}\partial_{x_j}\partial_{x_k}\psi+\psi\partial_{x_j}\partial_{x_k}\overline{\psi}
	+\partial_{x_j}\overline{\psi}\partial_{x_k}\psi+\partial_{x_j}\psi\partial_{x_k}\overline{\psi} \in C_b([0,\tau);L^q(\mathbb R^d))
\]
with $q:=\tfrac{4d}{4d-1}>1$. 

At this point, we apply Theorem 1 (ii) of \cite{lionsvillani} and conclude that 
\[
\nabla\sqrt{\rho}\in C_b([0,\tau);W^{1,2q}(\mathbb R^d)\,,\quad\text{ so that }(\partial_{x_j}\sqrt{\rho})(\partial_{x_k}\sqrt{\rho})\in C_b([0,\tau);L^q(\mathbb R^d))\,.
\]
Therefore \eqref{PreClo} can be put in the form
\[
m^2\rho^2u_ju_k=\rho\left(\tfrac14\hbar^2\partial_{x_j}\partial_{x_k}\rho-\hbar^2\left(\partial_{x_j}\sqrt{\rho}\right)\left(\partial_{x_k}\sqrt{\rho}\right)-\partial_{y_j}\partial_{y_k}\tilde{R}\Big|_{y=0}\right)\,,
\]
and we have seen that the term between parenthesis on the right-hand side above satisfies
\[
\tfrac14\hbar^2\partial_{x_j}\partial_{x_k}\rho-\hbar^2\left(\partial_{x_j}\sqrt{\rho}\right)\left(\partial_{x_k}\sqrt{\rho}\right)-\partial_{y_j}\partial_{y_k}\tilde{R}\Big|_{y=0}\in L^1_{loc}([0,\tau)\times\mathbb R^d_x)\,.
\]
Since $\rho>0$ a.e. in $\mathbb R_t\times\mathbb R^d_x$, we conclude that
\[
m^2\rho u_ju_k=\tfrac14\hbar^2\partial_{x_j}\partial_{x_k}\rho-\hbar^2\left(\partial_{x_j}\sqrt{\rho}\right)\left(\partial_{x_k}\sqrt{\rho}\right)-\partial_{y_j}\partial_{y_k}\tilde{R}\Big|_{y=0}\text{ a.e. in }\mathbb R_t\times\mathbb R^d_x\,.
\]
In particular $\rho u_ju_k\in L^1_{loc}([0,\tau)\times\mathbb R^d)$, and \eqref{Closure} holds as an equality between elements of $L^1_{loc}([0,\tau)\times\mathbb R^d)\subset\mathcal D'((0,\tau)\times\mathbb R^d)$.

The expression of $-\partial_{y_j}\partial_{y_k}\tilde{R}(t,x,0)$ in \eqref{Closure} is substituted in the left-hand side of \eqref{eq:euler2} to arrive at \eqref{eq:Madelung Euler}, which concludes the proof.
\end{proof}

\begin{remark}
Notice that our proof of Theorem \ref{T-Madelung} does not require knowing in advance that $\rho u_ju_k\in L^1_{loc}(\mathbb R_t\times\mathbb R^d)$. Indeed, since we already know that the quantity $\partial_{y_j}\partial_{y_k}\tilde{R}(t,x,0)$
belongs to $L^1_{loc}(\mathbb R_t\times\mathbb R^d)$, it is enough to compute it on the complement of a set of Lebesgue measure $0$, specifically on the complement of $\rho^{-1}(\{0\})$. 

However, one has the following information on $\rho u_ju_k$. Let us recall the following well-known formula for the current
\[
J_\ell=-\tfrac{i\hbar}{2m}(\overline\psi\partial_{x_\ell}\psi-\psi\partial_{x_\ell}\overline\psi)\,,
\]
so that
\[
\begin{aligned}
|J|^2=&\tfrac{\hbar^2}{4m^2}\sum_{\ell=1}^d\left|\overline\psi\partial_{x_\ell}\psi-\psi\partial_{x_\ell}\overline\psi\right|^2
\\
=&\tfrac{\hbar^2}{4m^2}\sum_{\ell=1}^d\left(2|\psi|^2|\partial_{x_\ell}\psi|^2-\psi^2\left(\partial_{x_\ell}\overline\psi\right)^2-\overline\psi^2\left(\partial_{x_\ell}\psi\right)^2\right)\,.
\end{aligned}
\]
At this point, we use the Sobolev embedding $H^1(\mathbb R^d)\subset L^p(\mathbb R^d)$ for all $p\in[2,\tfrac{2d}{d-2})$ if $d\ge 2$ while $H^1(\mathbb R)\subset L^\infty(\mathbb R)$, to find that
\[
\frac{|J|^2}\rho=\tfrac{\hbar^2}{4m^2}\sum_{\ell=0}^d\left(2|\partial_{x_\ell}\psi|^2-\frac\psi{\overline\psi}\left(\partial_{x_\ell}\overline\psi\right)^2-\frac{\overline\psi}\psi\left(\partial_{x_\ell}\psi\right)^2\right)
\in L^\infty(\mathbb R_t;L^p(\mathbb R^d_x))
\]
with $2\le p<\tfrac{2d}{d-2}$ if $d\ge 2$ and $p=\infty$ if $d=1$, since $\psi\in C_b(\mathbb R_t;H^2(\mathbb R^d))$ and $\psi\not=0$ a.e. on $\mathbb R_t\times\mathbb R^d$. Indeed, $\psi/\overline\psi$ and $\overline\psi/\psi$ are well-defined
measurable functions on the complement of $\psi^{-1}(\{0\})$, assumed to be a set of Lebesgue measure $0$, so that $\psi/\overline\psi$ and $\overline\psi/\psi\in L^\infty(\mathbb R_t\times\mathbb R^d_x)$ --- see next footnote. By the 
Cauchy-Schwarz inequality, this implies in turn that $\rho u_ju_k\in L^\infty(\mathbb R_t;L^p(\mathbb R^d_x))$ for the same values of $p$ as above.
\end{remark}

\begin{remark}\label{R-MadelungGasserMarko}
Madelung's original derivation of \eqref{MadelungSyst} is based on a representation of the wave function in the form $\psi(t,x)=\alpha(t,x)e^{im\phi(t,x)/\hbar}$ with real-valued $\alpha$ and $\phi$, and his momentum equation (equation (3') 
in \cite{Madelung1926}) takes the form of a Hamilton-Jacobi equation with a potential that is a linear combination of $V$ and $\Delta_x\alpha/\alpha$. Madelung's derivation implicitly assumes that $\psi\not=0$ everywhere. It seems that Gasser 
and Markowich \cite{GasserMarko} were the first to attempt deriving \eqref{MadelungSyst} as the system of moments in the momentum variable of the Wigner function of $\psi$. This is a very natural idea if one thinks of the Wigner equation 
as a bona fide kinetic equation. In particular, this approach leads to a momentum equation in \eqref{MadelungSyst} in conservation form, which is more natural from the hydrodynamic viewpoint, and may alleviate requirements on the positivity of 
the density function $|\psi|^2$. 

However the derivation of the Madelung system in \cite{GasserMarko} is incorrect and/or incomplete for the following reasons. 

Firstly, Lemma 2.1 in \cite{GasserMarko} does not make any assumption on the set $\psi^{-1}(\{0\})$, where $\psi$ is the solution of \eqref{SchroV} whose density function and momentum density are expected to satisfy the Madelung system. 
Yet the first step in the proof of Lemma 2.1 in \cite{GasserMarko} is a claim that $(\nabla_x\rho)^{\otimes 2}/\rho\in L^\infty(\mathbb R_t;L^1(\mathbb R^d))$, based on the fact that the ratio $\overline{\psi}/\psi$ is a well-defined element of 
$L^\infty(\mathbb R_t\times\mathbb R^d)$, a statement known to be true if and only if\footnote{Obviously, if $\psi$ is Borel measurable on $\mathbb R^d$ and $\psi\not=0$ a.e., then $\overline\psi/\psi$ is Borel measurable on the complement 
of the set $\mathcal N:=\psi^{-1}(\{0\})$ which is of Lebesgue measure $0$, so that $\overline\psi/\psi$ is measurable on $\mathbb R^d$, and $|\overline\psi/\psi|=1$ a.e. on $\mathbb R^d$ so that $\overline\psi/\psi\in L^\infty(\mathbb R^d)$.
If $\mathcal N$ is of Lebesgue measure $|\mathcal N|>0$, one has $\psi=|\psi|u$ with
\[
u=\frac{\psi+\mathbf 1_\mathcal Ne^{i\Theta}}{|\psi|+\mathbf 1_\mathcal N}\,,\qquad|u|=1\,,\quad\text{ so that one can set }
\frac{\overline\psi}{\psi}:=\frac{\overline u}{u}=\frac{\overline{\psi}+\mathbf 1_\mathcal Ne^{i\Theta}}{\psi+\mathbf 1_\mathcal Ne^{i\Theta}}=\frac{\overline\psi}\psi\mathbf 1_{\mathcal N^c}+e^{-2i\Theta}\mathbf 1_\mathcal N\,,
\]
where $\Theta$ is an arbitrary real-valued function. Here again $|\overline\psi/\psi|=1$, but $\overline\psi/\psi$ is not a well-defined element of $L^\infty(\mathbb R^d)$, since this definition of $\overline\psi/\psi$ explicitly involves the function 
$\Theta$ which is arbitrary on the set $\mathcal N$ of positive Lebesgue measure. One could even choose a non measurable $\Theta$, in which case $\overline\psi/\psi\notin L^\infty(\mathbb R^d)$.} the set $\psi^{-1}(\{0\})$ is of Lebesgue 
measure $0$. 

Secondly, even if this missing assumption is added to the statement of Lemma 2.1 in \cite{GasserMarko}, the remaining part of its proof is left to the reader as an ``easy exercise'' --- in the words of the authors of \cite{GasserMarko}. 
Deriving the local conservation law of momentum \eqref{eq:euler2} from \eqref{SchroV} is classical --- even though the second term on the right-hand side is more often written in the equivalent form
\[
-\partial_{y_j}\partial_{y_k}\tilde R\big|_{y=0}=\int_{\mathbb R^d}\xi_j\xi_kW_\hbar[R]d\xi\,.
\]
Starting from \eqref{eq:euler2}, the closure relation \eqref{Closure} is an almost straightforward consequence of our identity \eqref{IdentBisPure}, already used in the proof of our Proposition 2.

Another merit of this approach is that it leads to a structure of the Madelung pressure tensor (see \eqref{DefQuantPress} below) involving the expression
\[
\nabla_x^2\rho-4(\nabla_x\sqrt{\rho})^{\otimes 2}\quad\text{ instead of }\quad\rho\nabla_x^2\ln\rho\,,
\]
which is seen to be a $L^1_{loc}$ function as a consequence of the Lions-Villani observation \cite{lionsvillani} on the Sobolev regularity of square roots of nonnegative functions. This property seems less obvious on
the expression $\rho\nabla_x^2\ln\rho$ more commonly found in the literature as in \cite{Madelung1926}.

Finally, allowing $\psi(t,\cdot)$ to vanish on a set of measure $0$ is not only a minor technical improvement over the case where $\psi(t,x)\not=0$ for all $x\in\mathbb R^d$. There are situations of physical interest involving nonlinear 
variants of \eqref{SchroV} and isolated points $x_j(t)$ where $\psi(t,x_j(t))=0$, such as vortex dynamics in the Ginzburg-Landau equation, where Madelung's quantum hydrodynamics appear naturally: see for instance \cite{LinXin}.
\end{remark}

\smallskip
If moreover $\rho\in C^2([0,\tau)\times\mathbb R^d_x)$ is such that $\rho>0$ everywhere on $[0,\tau)\times\mathbb R^d_x$, there is another equivalent form of the Euler equation \eqref{eq:Madelung Euler}, 
\begin{equation}\label{eq:Madelung Euler QuantPress}
\partial_t(\rho u)+\nabla_x\cdot(\rho (u\otimes u+\tfrac1m\Pi))=-\tfrac1m\rho\nabla_xV,
\end{equation}
with pressure \emph{tensor}
\begin{equation}\label{DefQuantPress}
\Pi:=-\tfrac{\hbar^2}{4m}\nabla^2_x\ln\rho.
\end{equation}
Indeed, if $\rho\in C^2([0,\tau)\times\mathbb R^d_x)$ satisfies $\rho>0$ on $[0,\tau)\times\mathbb R^d_x$, one has
\[
\rho\nabla^2_x\ln\rho=\rho\nabla_x\left(\frac{\nabla_x\rho}{\rho}\right)=\nabla_x^2\rho-\frac{(\nabla_x\rho)^{\otimes 2}}{\rho}=\nabla_x^2\rho-4(\nabla_x\sqrt{\rho})^{\otimes 2}.
\]

Still another equivalent form of Madelung's system requires checking the formula for the Bohm quantum potential. Assuming that $\rho\in C^3([0,\tau)\times\mathbb R^d_x)$ is such that $\rho>0$ everywhere on 
$[0,\tau)\times\mathbb R^d_x$, observe that
\[
\begin{aligned}
\nabla_x\cdot(\rho\nabla^2_x\ln\rho)=&\nabla^2_x\ln\rho\cdot\nabla_x\rho+\rho\nabla_x\cdot\nabla^2_x\ln\rho
\\
=&\rho(\nabla^2_x\ln\rho\cdot\nabla_x\ln\rho+\nabla_x\Delta_x\ln\rho)
\\
=&\rho\nabla_x(\tfrac12|\nabla_x\ln\rho|^2+\Delta_x\ln\rho)
\\
=&2\rho\nabla_x\left(\left|\frac{\nabla_x\sqrt{\rho}}{\sqrt{\rho}}\right|^2+\nabla_x\cdot\left(\frac{\nabla_x\sqrt\rho}{\sqrt\rho}\right)\right)=2\rho\nabla_x\left(\frac{\Delta_x\sqrt\rho}{\sqrt\rho}\right)\,.
\end{aligned}
\]
Introducing the Bohm quantum potential
\begin{equation}\label{BohmQuantPot}
P:=-\tfrac{\hbar^2}{2m}\frac{\Delta_x\sqrt\rho}{\sqrt\rho}\,,
\end{equation}
the Euler equation in Madelung's system is recast as
\begin{equation}\label{eq:Madelung Euler BohmPot}
\partial_tu+u\cdot\nabla_xu=-\tfrac1m\nabla_x(P+V)\,.
\end{equation}

\smallskip
Finally, let us discuss the condition $\hbar^2\|\nabla\rho_\hbar\|^2_{L^2(\mathbb R^d)}\to 0$ in Proposition \ref{P-Monokinetic}. Let us compute
\[
\rho^2_\hbar\mathrm{Tr}(\nabla^2\ln\rho_\hbar)=\rho^2_\hbar\Delta\ln\rho_\hbar=2\rho^2_\hbar\nabla\cdot\left(\frac{\nabla\sqrt{\rho_\hbar}}{\sqrt{\rho_\hbar}}\right)
=2\rho^2_\hbar\left(\frac{\Delta\sqrt{\rho_\hbar}}{\sqrt{\rho_\hbar}}-\frac{|\nabla\sqrt{\rho_\hbar}|^2}{\rho_\hbar}\right),
\]
and
\[
\rho^2_\hbar\frac{\Delta\sqrt{\rho_\hbar}}{\sqrt{\rho_\hbar}}=\rho^{3/2}_\hbar\Delta\sqrt{\rho_\hbar}=\Delta(\rho^2_\hbar)-3\rho_\hbar|\nabla\sqrt{\rho_\hbar}|^2=\Delta(\rho^2_\hbar)-\tfrac34|\nabla\rho_\hbar|^2.
\]
These identites can be recast as
\[
\rho^2_\hbar\mathrm{Tr}\Pi_\hbar=\rho^2_\hbar P_\hbar+\tfrac{\hbar^2}{8m}|\nabla\rho_\hbar|^2
\quad\text{ and }\quad\rho^2_\hbar P_\hbar=-\tfrac{\hbar^2}{2m}\Delta(\rho^2_\hbar)+\tfrac38\tfrac{\hbar^2}m|\nabla\rho_\hbar|^2,
\]
or, equivalently
\[
\hbar^2|\nabla\rho_\hbar|^2=\tfrac43\hbar^2\Delta(\rho^2_\hbar)+\tfrac83m\rho^2_\hbar P_\hbar=\hbar^2\Delta(\rho^2_\hbar)+2m\rho^2_\hbar\mathrm{Tr}\Pi_\hbar.
\]

Returning to Proposition \ref{P-Monokinetic}, observe that the sequence $\hbar^2_n\rho^2_{\hbar_n}\to 0$ since $\rho_{\hbar_n}$ is bounded in $L^2(B(0,R))$. Hence
\[
\hbar^2_n\Delta(\rho^2_{\hbar_n})\to 0\quad\text{ in }\mathcal D'(\mathbb R^d)\qquad\text{ as }\hbar_n\to 0.
\]
Therefore
\[
\hbar_n^2\|\nabla\rho_{\hbar_n}\|^2_{L^2(B(0,R)}\to 0\iff\rho^2_{\hbar_n}P_{\hbar_n}\to 0\text{ in }\mathcal D'(\mathbb R^d)\iff\rho^2_{\hbar_n}\mathrm{Tr}\Pi_{\hbar_n}\to 0\text{ in }\mathcal D'(\mathbb R^d).
\]
In other words, the condition $\hbar_n^2\|\nabla\rho_{\hbar_n}\|^2_{L^2(B(0,R)}\to 0$ is equivalent to the fact that the quantum pressure $\Pi_{\hbar_n}$, or the Bohm potential $P_{\hbar_n}$, multiplied by 
$\rho_{\hbar_n}^2$ converges to $0$ in the sense of distributions on $\mathbb R^d$.

\begin{remark}
In view of the formulas \eqref{eq:restriction density}, \eqref{eq:restriction current} and \eqref{eq:restriction energy}, one can think of the Madelung system as equations for the moments in the $\xi$-variable
of $W_\hbar[\ketbra{\psi_\hbar}]$ in the classical (vanishing $\hbar$) limit. By comparison with the way in which fluid dynamical equations are derived from the kinetic theory of gases, one should think of
Wigner measures $w$, i.e. limit of converging subsequences $W_{\hbar_n}[\ketbra{\psi_{\hbar_n}}]$ in $\mathcal S'(\mathbb R^d\times\mathbb R^d)$, as the analogue of the local Maxwellian distribution
function in gas dynamics, i.e.
\[
\frac{\rho(t,x)}{(2\pi\theta(t,x))^{3/2}}\exp\left(-\frac{|v-u(t,x)|^2}{2\theta(t,x)}\right).
\]
For ideal gases, such as are described by the Boltzmann equation, the equation of state gives the pressure $p=\rho\theta$ in terms of the density $\rho$ and temperature $\theta$ (choosing units such that
the Boltzmann constant is $1$). Therefore, if $p(t,x)=0$ and $\rho(t,x)>0$, one has $\theta(t,x)=0$, in which case the Maxwellian must be replaced with
\[
\rho(t,x)\delta(v-u(t,x)).
\]
since $e^{-|v|^2/2\theta}/(2\pi\theta)^{3/2}\to\delta_0(v)$ in $\mathcal S'(\mathbb R^3)$ as $\theta\to 0^+$.

The situation described in Proposition \ref{P-Monokinetic} does not involve thermodynamical quantities such as a temperature. However the quantum pressure, or the Bohm potential, plays a role similar to
that of the classical pressure in gas dynamics: the vanishing of the quantum pressure implies that the limiting Wigner measure is monokinetic. However, this result is obtained in a completely different manner,
since one does not have any explicit formula representing the distribution of $\xi$ in $w$ in terms of the density function $\rho$, the current  density $J$ of the velocity field $u$ in \eqref{DefVelField} and the 
quantum pressure $\Pi$ in \eqref{DefQuantPress}.
\end{remark}

\noindent
\textbf{Acknowledgements.}
We wish to thank Shi Jin for introducing us to the problem discussed in this paper, and Norbert Mauser for suggesting several references closely related to our approach (especially the work of Gasser and
Markowich on Madelung's quantum hydrodynamics). We are also grateful to Thomas Alazard, Herbert Koch and Pierangelo Marcati for valuable insight on the vacuum problem in Madelung's equations.


\bibliographystyle{abbrv}
\bibliography{wigneraveraging}


\end{document}